\documentclass[11pt]{article}\textwidth 160mm\textheight 235mm
\usepackage{amssymb,amsmath}
\usepackage{amsfonts}
\usepackage{graphicx}
\usepackage{tikz}
\usepackage{cancel}
\usepackage{todonotes}
\usepackage{enumitem}
\usetikzlibrary{matrix}
\usetikzlibrary{arrows,shapes}
\newenvironment{proof}{\paragraph{Proof.}}{\hfill$\square$}
\numberwithin{equation}{section}
\numberwithin{figure}{section}
\newcommand{\argmax}{\operatornamewithlimits{arg\,max}}
\newcommand{\argmin}{\operatornamewithlimits{arg\,min}}

\newcommand{\n}[1]{\left\lVert#1\right\rVert}
\newcommand{\inp}[2]{\langle #1, #2 \rangle}
\newtheorem{Th}{Theorem}[section]

\newtheorem{Lemma}[Th]{Lemma}

\newtheorem{Def}[Th]{Definition}
\newtheorem{Example}[Th]{Example}
\def\tto{\rightrightarrows}
\def\zb{\bar{z}}
\def\vb{\bar{v}}
\def\xb{\bar{x}}

\def\yb{\bar{y}}
\def\pb{\bar{p}}

\def\lmt{\widehat{\lambda}}
\def\rge{{\rm rge\,}}
\def\st{\stackrel}

\def\Rm{\mathbb{R}^m}
\def\R{\mathbb{R}}
\def\Rn{\mathbb{R}^n}
\def\B{\mathbb{B}}
\def\O{\Omega}
\def\mcK{\mathcal{K}}

\def\dps{\displaystyle}
\def\p{\partial}
\def\dom{\mathrm{dom}\,}
\def\gph{\mathrm{gph}\,}
\def\Lcxb{\Lambda_{\mathrm{com}}(\xb)}

\def\lm{\lambda}
\def\th{\theta}
\def\al{\alpha}
\def\ph{\varphi}
\def\olm{\bar\lambda}
\def\thy{\theta_{Y,B}}
\def\thk{\theta_{\ok,B}}
\def\ox{\bar{x}}
\def\ov{\bar{v}}
\def\oy{\bar{y}}
\def\ow{\bar{w}}
\def\oz{\bar{z}}
\def\dd{\delta}
\def\ok{{\cal K}}
\def\d{{\mathrm d}}
\def\ve{\varepsilon}
\def\epsilon{\ve}
\def\dist{{\rm dist}}
\def\dom{{\rm dom}\,}
\def\span{{\rm span}}
\def\Lm{\Lambda}
\def\dn{\downarrow}
\def\oR{\overline{\mathbb{R}}}
\def\N{\mathbb{N}}
\def\emp{\emptyset}
\def\la{\langle}
\def\ra{\rangle}
\def\tilde{\widetilde}
\def\hat{\widehat}
\def\disp{\displaystyle}
\def\ss{\scriptsize }
\begin{document}
\begin{center}
\textbf{CRITICALITY OF LAGRANGE MULTIPLIERS\\IN EXTENDED NONLINEAR OPTIMIZATION}
\end{center}
\begin{center}
HONG DO\footnote{Department of Mathematics, Wayne State University, Detroit, Michigan, 48202, USA (fq0828@wayne.edu). Research of this author was partly supported by the USA National Science Foundation under grants DMS-1512846 and DMS-1808978, and by the USA Air Force Office of Scientific Research grant \#15RT04.}, BORIS S. MORDUKHOVICH\footnote{Department of Mathematics, Wayne State University, Detroit, Michigan, 48202, USA (boris@math.wayne.edu). Research of this author was partly supported by the USA National Science Foundation under grants DMS-1512846 and DMS-1808978, and by the USA Air Force Office of Scientific Research grant \#15RT04.} and M. EBRAHIM SARABI\footnote{Department of Mathematics, Miami University, Oxford, Ohio, 45056, USA (sarabim@miamioh.edu).}
\end{center}
\providecommand{\Abstract}[1]{\textbf{Abstract. }#1}
\Abstract{The paper is devoted to the study and applications of criticality of Lagrange multipliers in variational systems, which are associated with the class of problems in composite optimization known as extended nonlinear programming (ENLP). The importance of both ENLP and the concept of multiplier criticality in variational systems has been recognized in theoretical and numerical aspects of optimization and variational analysis, while the criticality notion has never been investigated in the ENLP framework. We present here a systematic study of critical and noncritical multipliers in a general variational setting that covers, in particular, KKT systems in ENLP with establishing their verifiable characterizations as well as relationships between noncriticality and other stability notions in variational analysis. Our approach is mainly based on advanced tools of second-order variational analysis and generalized differentiation.}\\[1ex]

\section{Introduction}\label{intro}

One of the major goals of this paper is to study a remarkable class of optimization problems given in the following, formally unconstrained, {\em composite format}:
\begin{equation}\label{co}
\textrm{minimize }\;\ph(x):=\ph_0(x)+\theta\big(\Phi(x)\big),\quad x\in\R^n,
\end{equation}
where $\ph_0\colon\R^n\to\R$ is an original cost function and $\Phi\colon\R^n\to\R^m$ is a constraint mapping, both are twice differentiable at the reference points unless otherwise stated, and where $\th\colon\R^m\to\oR:=(-\infty,\infty]$ is an extended-real-valued function defined for all $u\in\R^m$ by the formula
\begin{equation}\label{theta}
\th(u)=\thy(u):=\underset{y\in Y}\sup{\Big\{\inp{y}{u}-\frac{1}{2}\inp{y}{By}\Big\}}
\end{equation}
via a convex polyhedral set $Y:=\{y\in\mathbb{R}^m\arrowvert\;\inp{b_i}{y}\le\alpha_i,\;i=1,\ldots,p\}$ as well as an $m\times m$ positive-semidefinite and symmetric  matrix $B$.

Note that the unconstrained composite format \eqref{co} gives us a convenient representation of the constrained optimization problem to minimize the cost function $\ph_0(x)$ subject to the inclusion constraint $\Phi(x)\in\Theta:=\{u\in\R^m|\;\th(u)<\infty\}$. In particular, conventional nonlinear programs (NLPs) with $s$ inequality constraints and $m-s$ equality constraints described by ${\cal C}^2$-smooth functions can be written in the composite format \eqref{co}, where $\th:=\dd_\Theta$ is the indicator function of the polyhedron $\Theta:=\R^s_-\times\{0\}^{m-s}$ that is equal to $0$ on $\Theta$ and to $\infty$ otherwise.

Problems of the ENLP type \eqref{co} with $\th$ given by \eqref{theta} were introduced by Rockafellar \cite{r} under the name of {\em extended nonlinear programs} (ENLPs). It has been realized over the years that ENLPs in this form provide a suitable framework for developing both theoretical and computational aspects of optimization in broad classes of constrained problems that include stochastic programming, robust optimization, etc. The special expression (\ref{theta}) for the extended-real-valued function $\th$, known as the {\em dualizing representation} or the {\em piecewise linear-quadratic penalty}, is significant for the theory and applications of Lagrange multipliers in the Karush-Kuhn-Tucker (KKT) systems associated with the ENLPs under consideration.

It is not hard to check (see more details in Section~\ref{comp-opt}) that KKT systems associated with local optimal solutions to ENLPs
are included in the following more general class of {\em variational systems} of the {\em subdifferential type}
\begin{equation}\label{VS}
\Psi(x,\lm):=f(x)+\nabla\Phi(x)^*\lm=0,\;\lm\in\p\th\big(\Phi(x)\big)\;\mbox{ with }\;\th=\thy,
\end{equation}
where $f\colon\R^n\rightarrow\R^n$ is a differentiable mapping while $\Phi\colon\R^n\rightarrow\R^m$ is a twice differentiable mapping in the classical sense \cite[Definition~13.1(i)]{rw},
 where $\thy$ is taken from \eqref{theta}, where $^*$ indicates the matrix transposition/adjoint operator, and where $\p$ stands for the subdifferential of convex analysis.

The main attention of this paper is paid to a systematic study of the  {\em multiplier criticality} concept (i.e., the notions of critical and noncritical Lagrange multipliers) for variational systems of type \eqref{VS} with applications to KKT systems in ENLPs.

The notions of critical and noncritical multipliers were first introduced by Izmailov \cite{iz05} for the classical KKT systems corresponding to {\em NLPs with equality constraints} described by ${\cal C}^2$-smooth functions. It has been realized from the very beginning that the presence of critical multipliers plays a {\em negative} role in numerical optimization and is largely responsible for primal slow convergence in primal-dual algorithms of the Newtonian type. Further strong developments in this direction for NLPs and related variational inequalities have been done over the years, mainly by Izmailov, Solodov, and their collaborators; see, e.g., the book \cite{is14} and the survey paper \cite{is15}, which is entirely devoted to critical multipliers. The criticality definitions in the above publications are heavily based on the specific structures of NLPs and related variational inequalities.

In \cite{ms17}, Mordukhovich and Sarabi suggested new definitions of critical and noncritical multipliers for a general class of {\em subdifferential variational systems} of type \eqref{VS}, where $\th$ may be even a nonconvex extended-real-valued function. The given definitions in \cite{ms17} are expressed via second-order generalized differential constructions of variational analysis while reduced to those from \cite{iz05,is14} for the classical KKT systems corresponding to NLPs. Furthermore, for extended-real-valued {\em convex piecewise linear} (CPWL) functions $\th$ in \eqref{VS}, which include \eqref{theta} when $B=0$, the definitions of critical and noncritical multipliers are expressed in \cite{ms17} entirely in terms of the problem data with the subsequent characterizations of criticality and various applications to optimization and stability problems for such systems.

The quite recent paper of the same authors \cite{ms18} contains counterparts of some major results from \cite{ms17} with developing also novel issues on criticality for variational systems described by
\begin{equation}\label{VS1}
f(x)+\nabla\Phi(x)^*\lm=0,\;\lm\in N_\Theta\big(\Phi(x)\big),
\end{equation}
where $f$ and $\Phi$ are the same as in \eqref{VS}, and where $N_\Theta$ is the normal cone to a {\em ${\cal C}^2$-cone reducible} set $\Theta\subset\R^m$. This framework covers, in particular, KKT systems associated with general problems of (nonpolyhedral) {\em conic programming}; see, e.g., \cite{bs}.\vspace*{0.05in}

The main results of the current paper extend those from \cite{ms17}, obtained for CPWL functions $\th$, to the case of functions $\thy$ defined in \eqref{theta}, which form a major class of extended-real-valued convex {\em piecewise linear-quadratic} functions in variational analysis; see \cite{rw} and Section~\ref{prel} below. At the same time, the new results obtained here are completely independent from those derived for the variational system \eqref{VS1} in \cite{ms17} in the case of nonpolyhedral sets $\Theta$ therein.

The basic tools of first-order and second-order generalized differentiation employed in this paper are {\em tangentially generated}, except the classical subdifferential of convex analysis. We mostly rely on the generalized differential theory in primal spaces developed by Rockafellar; see \cite{rw} and the references therein. Using these tools allows us to establish verifiable characterizations of noncritical multipliers in the general setting of \eqref{VS}, to characterize the uniqueness of Lagrange multipliers in \eqref{VS}, to ensure noncriticality for ENLPs via a new second-order optimality condition, which is employed in turn to verify the important stability property of solutions to KKT systems that is known as robust isolated calmness and is related to noncriticality. We also reveal a relationship between the isolated calmness and Lipschitz-like properties of solution maps for canonically perturbed variational systems with the piecewise linear-quadratic term \eqref{theta}.

As mentioned above, the existence of critical multipliers is a negative factor in convergence analysis, since it seems to prevent primal superlinear convergence of major primal-dual algorithms. Thus it is crucial to find verifiable conditions, expressed entirely in terms of the problem data in question, which ensure that critical multipliers corresponding to this minimizer do not arise. It is conjectured in \cite{m15}, based on preliminary results for NLPs, that {\em full stability} of local minimizers in the sense of \cite{lpr} {\em rules out} the appearance of {\em critical multiplies}. This conjecture was verified in \cite{ms17} for polyhedral problems of type \eqref{co} with convex piecewise linear functions $\th$. Now we justify this conjecture in the general case of ENLPs with piecewise linear-quadratic functions $\thy$ in form \eqref{theta}. \vspace*{0.05in}

The rest of the paper is organized as follows. In Section~\ref{prel} we present some definitions and facts from variational analysis and generalized differentiation that are broadly employed throughout the whole paper. Other variational constructions and results are recalled in those places of the subsequent sections where they are actually used.

Section~\ref{crit-def} contains basic {\em definitions} of {\em critical} and {\em noncritical multipliers} for variational systems \eqref{VS} involving piecewise linear-quadratic functions of type \eqref{theta} with providing equivalent descriptions, examples, and discussions. In Section~\ref{unique} we obtain new results on the relationship between the well-recognized {\em calmness} and {\em isolated calmness} properties of multiplier maps associated with the variational systems \eqref{VS} with the piecewise linear-quadratic term \eqref{theta} and the {\em uniqueness} of Lagrange multipliers in such systems. This is certainly of its independent interest, while the developed approach and results can be viewed as the preparation to the subsequent characterizations of noncritical multipliers in the variational systems under consideration.

Section~\ref{noncrit-char} plays a central role in the paper. It establishes major {\em characterizations} of {\em noncritical multipliers} for systems \eqref{VS} with $\thy$ taken from \eqref{theta} via a novel {\em semi-isolated calmness} property for solution maps to canonical perturbations of \eqref{VS} and also via two new {\em error bounds} that are specific for the variational systems \eqref{VS} with the piecewise linear-quadratic term
\eqref{theta}.

Section~\ref{comp-opt} is devoted to noncritical multipliers in {\em KKT systems} associated with {\em ENLPs} for which the results of the previous sections are automatically applied with the specification of $\Psi$ in \eqref{VS} as the $x$-partial gradient of the appropriate Lagrangian. The main new result here, that is characteristic to the optimization framework, is a novel {\em second-order sufficient condition} for strict local minimizers, which also ensures that all the corresponding multipliers are noncritical.

In Section~\ref{full-stab} we justify, for the case of ENLPs from \eqref{co} and \eqref{theta}, the aforementioned conjecture on {\em excluding critical multipliers} corresponding to a {\em fully stable} local minimizer for the given ENLP. The proof of this result is based on characterizations of noncriticality via semi-isolated calmness obtained in Section~\ref{noncrit-char}.

The last Section~\ref{lip-stab} provides applications of the developed characterizations of noncritical multipliers for the variational systems under consideration to the study of an important stability property of solution maps to KKT systems associated with ENLPs. This property of set-valued mappings has been recently recognized as {\em robust isolated calmness}. The results obtained above allow us to characterize robust isolated calmness via the noncriticality and uniqueness of Lagrange multipliers on one side and via the new second-order optimality condition for ENLPs on the other.
Finally, we characterize the Lipschitz-like/Aubin property of solution maps to perturbed variational systems and establish its relationship with isolated calmness.

\section{Preliminaries from Variational Analysis}\label{prel}

In this section we review, based on the book \cite{rw}, some basic notions of generalized differentiation in variational analysis and then recall important facts broadly used in what follows. Throughout the paper we use the standard notation of variational analysis; see \cite{m18,rw}.

Given a nonempty subset $\O\subset\R^d$ and a point $\oz\in\O$, the (Bouligand-Severi) {\em tangent/contingent cone} $T_\Omega(z)$ to $\Omega$ at $\oz$ is defined by
\begin{equation}\label{tan}
T_\Omega(\oz):=\Big\{w\in\R^d\Big|\;\exists\,z_k\xrightarrow{\Omega}\oz,\;\exists\,\alpha_k\ge 0\;\textrm{ with }\;\alpha_k(z_k-z)\rightarrow w\;\textrm{ as }\;k\rightarrow\infty\Big\},
\end{equation}
where the symbol $z\st{\O}{\to}\oz$ indicates that $z\to\oz$ with $z\in\O$.

For a set-valued mapping $F\colon\R^n\rightrightarrows\R^p$, define its {\em domain} and {\em graph} by, respectively,
\begin{equation*}
\dom F:=\big\{x\in\R^n\big|\;F(x)\ne\emp\big\}\;\mbox{ and }\;\gph F:=\big\{(x,y)\in\R^n\times\R^p\big|\;y\in F(x)\big\}.
\end{equation*}
The {\em graphical derivative} of $F$ at $(\ox,\oy)\in\gph F$ is given by
\begin{equation}\label{gr-der}
DF(\xb,\yb)(u):=\big\{v\in\R^p\big|\;(u,v)\in T_{\gph F}(\xb,\yb)\big\},\quad u\in\R^n.
\end{equation}

Next we consider an extended-real-valued function $\ph\colon\R^n\to\oR:=(-\infty,\infty]$ with $\ox\in\dom\ph:=\{x\in\R^n|\;\ph(x)<\infty\}$.
Given $\oy\in\R^n$, the {\em second subderivative} of $\ph$ at $(\ox,\oy)$ in the direction $\ow$ is defined by
\begin{equation}\label{ssd1}
\d^2\ph(\bar x,\oy)(\ow):=\liminf_{\substack{t\dn 0\\w\to\ow}}\dfrac{\ph(\ox+tw)-\ph(\ox)-t\langle\oy,\,w\rangle}{\frac{1}{2}t^2}.
\end{equation}
When $\ph$ is convex and proper (i.e., $\dom\ph\ne\emp$), we use its {\em subdifferential} (i.e., the collection of subgradients) at $\ox\in\dom\ph$ given by
\begin{equation}\label{sub}
\partial\ph(\ox):=\big\{v\in\R^n\big|\;\la v,x-\ox\ra\le\ph(x)-\ph(\ox)\;\mbox{ for all }\;x\in\R^n\big\}.
\end{equation}
If $\O\subset\R^n$ is a nonempty convex set, then the {\em normal cone} to $\O$ at $\ox\in\O$ is the subdifferential \eqref{sub} of its indicator function and thus is defined by
\begin{equation}\label{nc-conv}
N_\O(\ox):=\big\{v\in\R^n\big|\;\la v,x-\ox\ra\le 0\;\mbox{ for all }\;x\in\O\big\}.
\end{equation}
The {\em critical cone} to $\O$ at $\ox$ for $\ov\in N_\O(\ox)$ is expressed via the tangent cone \eqref{tan} as
\begin{equation}\label{cri3}
K_\O(\ox,\ov):=T_\O(\ox)\cap\{\ov\}^\perp
\end{equation}
with the notation $\{\ov\}^\perp:=\big\{w\in\R^n|\;\la w,v\ra=0\}$.

Along with \eqref{ssd1}, we employ in this paper yet another second-order generalized derivative of an extended-real-valued convex function $\ph\colon\R^n\to\oR$ at $\ox\in\dom\ph$ for $\ov\in\partial\ph(\ox)$ that is defined via the graphical derivative \eqref{gr-der} of the subgradient mapping $\partial\ph\colon\R^n\tto\R^n$ under the name of the {\em subgradient graphical derivative} by
\begin{equation}\label{sgd}
D\partial\ph(\ox,\ov)(u):=D\big(\partial\ph\big)(\ox,\ov)(u),\quad u\in\R^n.
\end{equation}

Invoking the constructions above, we now formulate the basic facts about the functions $\thy$ taken from \eqref{theta} that are systematically exploited in the paper. The proofs of these facts can be found in \cite[Examples~11.18, 13.23 and Theorem~13.40]{rw}. Recall that the {\em horizon cone} of a nonempty set $Y\subset\R^m$ used below is defined by
\begin{equation*}
Y^\infty:=\big\{y\in\R^m\big|\;\exists\,y_k\in Y,\;\exists\,\lm_k\dn 0\;\mbox{ with }\;\lm_k y_k\to y\big\}.
\end{equation*}
Recall also \cite[Definition~10.20]{rw} that a function $\ph\colon\R^n\to\oR$ is {\em piecewise linear-quadratic} if its domain $\dom\ph$  can be represented as the union of finitely many convex polyhedral sets, relative to each of which $\ph(x)$ is given by an
expression of the form $\frac{1}{2}\la x,Ax\ra+\la a,x\ra+\al$ for some scalar $\al\in\R$, vector $a\in\R^n$, and $n\times n$ symmetric matrix $A$.

\begin{Th}{\bf(properties of piecewise linear-quadratic penalties).}\label{gdt} Let $\thy$ be defined by \eqref{theta}. Then the following properties hold:\vspace{-0.2 cm}
\begin{itemize}[noitemsep]
\item[\bf(i)] The function $\thy$ is a proper and convex piecewise linear-quadratic with the domain
\begin{equation*}
\dom\thy=\big(Y^\infty\cap\ker B\big)^*.
\end{equation*}
\item[\bf(ii)] The subdifferential \eqref{sub} of $\thy$ is calculated by
\begin{equation}\label{fo}
\partial\thy(u)=\argmax_{y\in Y}\big\{\inp{y}{u}-\frac{1}{2}\inp{y}{By}\big\}=(N_Y+B)^{-1}(u),\quad u\in\R^m.
\end{equation}
\item[\bf(iii)] Given any $(\oz,\olm)\in\gph\partial\thy$, the second subderivative \eqref{ssd1} is calculated by
\begin{equation}\label{ssd}
\d^2\thy(\zb,\olm)(u)=2\thk(u):=\sup_{w\in\ok}\big\{2\inp{w}{u}-\inp{w}{Bw}\big\},\quad u\in\R^m,
\end{equation}
in the same form $\thk(u)$ as in \eqref{theta} with the replacement of $Y$ by critical cone $\ok:=K_Y(\olm,\zb-B\olm)$ defined via \eqref{cri3}. Furthermore, the subgradient graphical derivative \eqref{sgd} of $\thy$ at $\oz$ for $\bar \lm$ is represented as
\begin{equation}\label{gdr}
D\partial\thy(\bar{z},\bar{\lm})(u)=\partial\thk(u),\quad u\in\R^m.
\end{equation}
\end{itemize}
\end{Th}

\section{Multiplier Criticality in Piecewise Linear-Quadratic Settings}\label{crit-def}

In this section we formulate the definitions of critical and noncritical multipliers corresponding to stationary points of the variational system \eqref{VS} with the piecewise linear-quadratic term \eqref{theta}, establish
an equivalent description of criticality entirely via the given data of \eqref{VS}, and then present two examples illustrating the calculation of critical and noncritical multipliers for this setting.

Given a point $\bar{x}\in\mathbb{R}^n$, define the set of {\em Lagrange multipliers} associated with $\ox$ by
\begin{equation}\label{lag}
\Lambda(\bar{x}):=\big\{\lm\in\mathbb{R}^m\big|\;\Psi(\bar{x},\lm)= 0,\;\lm\in\partial\thy\big(\Phi(\bar{x})\big)\big\}.
\end{equation}
If $(\ox,\olm)$ is a solution to the variational system \eqref{VS}, we clearly get $\olm\in\Lambda(\ox)$. Furthermore, it is not hard to check that the inclusion $\olm\in\Lambda(\ox)$ ensures that $\ox$ is a {\em stationary point} of \eqref{VS} in the sense that it satisfies
the condition
\begin{equation}\label{stat}
0\in f(\bar{x})+\partial\big(\thy\circ\Phi\big)(\bar{x}).
\end{equation}
Suppose from now on that $\Lambda(\xb)\ne\emp$, which is ensured, e.g., by any constraint qualification condition in problems of constrained optimization. The following definitions of critical and noncritical multipliers for \eqref{VS},
are just specifications of those from \cite{ms17}, given there for general variational systems with the subsequent implementation for the case of a convex piecewise linear function $\th$.
It is worth noticing that the function $\th$ from \eqref{theta} with $B=0$ is  convex piecewise linear, namely its epigraph is a convex polyhedral set, and so
can be covered by the results already established in \cite{ms17}; however, when $B\neq 0$, it is a convex  piecewise linear-quadratic function and requires different techniques
to achieve similar results.

\begin{Def}{\bf(critical and noncritical multiplies in variational systems).}\label{crit} Let $(\ox,\olm)$ be  a solution to the variational system \eqref{VS}. We say that $\bar{\lm}\in\Lambda(\bar{x})$ is a {\sc critical Lagrange multiplier} for \eqref{VS} corresponding to $\ox$ if there exists a nonzero vector $\xi\in\mathbb{R}^n$ such that
\begin{equation}\label{cm}
0\in\nabla_x\Psi(\bar{x},\bar{\lm})\xi+\nabla\Phi(\bar{x})^*D\partial\thy\big(\Phi(\bar{x}),\bar{\lm}\big)\big(\nabla\Phi(\bar{x}\big)\xi).
\end{equation}
A given multiplier $\bar{\lm}\in\Lambda(\bar{x})$ is {\sc noncritical} for \eqref{VS} corresponding to $\ox$ if the generalized equation \eqref{cm} admits only the trivial solution $\xi=0$.
\end{Def}

Applying the representations of Theorem~\ref{gdt} for the graphical derivative in \eqref{cm} gives us an equivalent description of critical and noncritical multipliers from Definition~\ref{crit},
 expressed entirely in terms of the initial data of \eqref{VS}.

\begin{Th}{\bf(equivalent description of criticality via piecewise linear-quadratic penalties).}\label{desc} Let $(\bar{x},\bar{\lm})$ be a solution to the variational system \eqref{VS} with the term $\thy$ taken from \eqref{theta}. Denoting $\oz:=\Phi(\ox)$ and $\ok:=K_Y(\olm,\zb-B\olm)$ via the critical cone \eqref{cri3}, we have that the multiplier $\bar{\lm}$ corresponding to $\ox$ is critical for \eqref{VS} if and only if the system of relationships
\begin{equation}\label{cri}
\begin{cases}
\nabla_x\Psi(\bar{x},\bar{\lm})\xi+\nabla\Phi(\bar{x})^*\eta=0,\quad\inp{\nabla\Phi(\bar{x})\xi-B\eta}{\eta}=0,\\
\nabla\Phi(\bar{x})\xi-B\eta\in\ok^*,\;\mbox{ and }\;\eta\in\ok
\end{cases}
\end{equation}
admits a solution $(\xi,\eta)\in\mathbb{R}^n\times\mathbb{R}^m$ with $\xi\ne 0$. Accordingly, $\olm$ is a noncritical multiplier in this setting if and only if we have $\xi=0$ for any solution $(\xi,\eta)$ to \eqref{cri}.
\end{Th}
\begin{proof}  To achieve the claimed equivalencies, we require to calculate the graphical derivative $D\partial\thy$ in \eqref{cm} for the function $\thy$ given in \eqref{theta}. First we use formula \eqref{gdr} from Theorem~\ref{gdt}(iii), which yields
\begin{equation*}
D\partial\thy(\bar{z},\bar{\lm})\big(\nabla\Phi(\bar{x})\xi\big)=\partial\thk\big(\nabla\Phi(\bar{x})\xi\big).
\end{equation*}
On the other hand, the second expression of $\partial\thk$ in \eqref{fo} of Theorem~\ref{gdt}(ii) shows that
\begin{equation*}
\partial\thk\big(\nabla\Phi(\bar{x})\xi\big)=\big(N_{\ok}+B\big)^{-1}\big(\nabla\Phi(\bar{x})\xi\big).
\end{equation*}
Putting these representations together, we arrive at
\begin{equation}\label{desc1}
D\partial\thy(\bar{z},\bar{\lm})\big(\nabla\Phi(\bar{x})\xi\big)=\big(N_{\ok}+B\big)^{-1}\big(\nabla\Phi(\bar{x})\xi\big).
\end{equation}
Picking further any vector $\eta$ from the set on the left-hand side of \eqref{desc1} gives us therefore that
$\eta\in(N_{\ok}+B\big)^{-1}\big(\nabla\Phi(\bar{x})\xi)$ and so $\nabla\Phi(\ox)\xi-B\eta\in N_{\ok}(\eta)$. Since $\ok$ is a convex cone, the latter inclusion is equivalent to the conditions
\begin{equation*}
\inp{\nabla\Phi(\bar{x})\xi- B\eta}{\eta}=0,\quad\nabla\Phi(\bar{x})\xi-B\eta\in\ok^*,\quad \eta\in\ok.
\end{equation*}
Finally, we substitute the obtained descriptions of $\eta\in D\partial\thy(\bar{z},\bar{\lm})(\nabla\Phi(\bar{x})\xi)$ into \eqref{cm} and thus clearly verify both assertions of the theorem.
\end{proof}\vspace*{0.05in}

Next we present two examples, which demonstrate how to use the descriptions of Theorem~\ref{desc} to explicitly determine critical and noncritical multipliers and illustrate in this way some characteristic features of multiplier criticality.

\begin{Example}{\bf(calculating critical and noncritical multipliers).}\label{ex1}
{\rm Consider the multidimensional case of \eqref{VS} with $\thy$ from \eqref{theta}, where $B=I_m=:I$ is the $m\times m$ identity matrix, and where the convex polyhedral set $Y$ is the nonnegative orthant in $\R^m$, i.e.,
\begin{equation*}
Y=\R^m_+:=\big\{y=(y_1,\ldots,y_m)\in\R^m\big|\;y_i\ge 0\;\mbox{ for all }\;i=1,\ldots,m\big\}.
\end{equation*}
Thus the function $\theta_{Y,B}$ from \eqref{theta} reduces in this case to
\begin{equation*}
\dps\theta_{\Rm_+,I}(u)=\sup_{y\in\Rm_+}\Big\{\inp{y}{u}-\frac{1}{2}\inp{y}{y}\Big\},\quad u\in\R^m.
\end{equation*}
For any $\xb\in\Rn$ and $\zb:=\Phi(\xb)$, by Theorem~\ref{gdt}(ii) we have that $\lm\in\partial\theta_{\Rm_+,I}(\zb)$ if and only if $\zb-B\lm\in N_{\Rm_+}(\lm)=\Rm_-\cap\lm^\perp$. Denoting $\zb-\lm$ by $\lmt$, the latter inclusion is equivalent to the following system
of equations and inclusions:
\begin{equation}\label{sgr}
\begin{cases}
\lm+\lmt=\zb\\
\inp{\lm}{\lmt}=0\\
\lm\in\Rm_+\\
\lmt\in\Rm_-
\end{cases}
\end{equation}
It is not hard to see that for each fixed $\xb$ and $\zb=\Phi(\xb)$ this system has only one solution, which implies that the set of Lagrange multipliers has at most one element.

We now give two specific examples of mappings $f$ and $\Phi$, where one has a noncritical multiplier and the other has a critical multiplier. First, let $f(x):=x$ and $\Phi(x):=(x_1,0,\ldots,0)\in\Rm$ for all $x=(x_1,\ldots,x_n)\in\Rn$, and let $\xb:=0\in\Rn$. Combining \eqref{sgr} with the fact that $\Psi(\xb,\lm)=(\lm_1,0,\ldots,0)\in\Rn$ implies that the unique Lagrange multiplier is $\olm=0$. Then we calculate the critical cone $\ok=K_Y(0,\oz)$ in Theorem~\ref{desc} with $\zb=\Phi(\xb)=0$ and its dual cone $\ok^*$ by, respectively,
\begin{equation*}
\ok=T_{\Rm_+}(0)\cap\{\zb\}^\perp=\Rm_+\;\mbox{ and }\;\ok^*=\span\{\zb\}+N_{\Rm_+}(0)=\Rm_-.
\end{equation*}
It follows from Theorem~\ref{desc} that the unique Lagrange multiplier $\olm=0$ is noncritical if and only if the system of equations and inclusions
\begin{equation*}
\begin{cases}
\nabla_x\Psi(\bar{x},\bar{\lambda})\xi+\nabla\Phi(\bar{x})^*\eta=0\\
\inp{\nabla\Phi(\bar{x})\xi-\eta}{\eta}=0\\
\nabla\Phi(\bar{x})\xi-\eta\in\Rm_-\\
\eta\in\Rm_+
\end{cases}
\end{equation*}
admits the only solution pairs $(\xi,\eta)\in\Rn\times\Rm$ with $\xi=0$. Denoting $\zeta:=\nabla\Phi(\xb)\xi-\eta$, the above system can be equivalently rewritten as
\begin{equation}\label{crsys}
\begin{cases}
\nabla_x\Psi(\bar{x},\bar{\lambda})\xi+\nabla\Phi(\bar{x})^*\eta=0\\
\nabla\Phi(\xb)\xi-\eta-\zeta=0\\
\inp{\zeta}{\eta}=0\\
\zeta\in\Rm_-\\
\eta\in\Rm_+.
\end{cases}
\end{equation}
Since $\nabla_x\Psi(\xb,\olm)\xi=\xi$, $\nabla\Phi(\xb)\xi=(\xi_1,0,\ldots,0)\in\Rm$, and $\nabla\Phi(\xb)^*\eta=(\eta_1,0,\ldots,0)\in\Rn$ for any $\eta=(\eta_1,\ldots,\eta_m)\in\Rm$, it can be easily checked that the latter system has the unique solution pair $(\xi,\eta)=(0,0)$. This tells us that $\olm=0$ is a noncritical multiplier.

Next we consider the case where $\Phi(x):=(x_1,0,\ldots,0)\in\Rm$ as before while $f(x):=(x_1,\ldots,x_{n-1},0)\in\Rn$ for all $x=(x_1,\ldots,x_n)\in\Rn$. Proceeding similarly to the previous case shows that $\olm=0$ is the unique Lagrange multiplier with the same critical cone $\mcK$. In this setting we have $\nabla_x\Psi(\xb,\olm)\xi=(\xi_1,\ldots,\xi_{n-1},0)\in\Rn$, and therefore system \eqref{crsys} reduces to
\begin{equation*}
\begin{cases}
(\xi_1,\ldots,\xi_{n-1},0)+(n_1,0,\ldots,0)=0\\
\nabla\Phi(\xb)\xi-\eta-\zeta=0\\
\inp{\zeta}{\eta}=0\\
\zeta\in\Rm_-\\
\eta\in\Rm_+.
\end{cases}
\end{equation*}
It shows that all the pairs $(\xi,\eta)$ with $\eta=0$ and $\xi=(0,\ldots,0,\xi_n)$ for $\xi_n\in\R$ are solutions to the above system. Thus the multiplier $\olm=0$ is critical.

In Section~\ref{comp-opt} we  revisit this example in the optimization framework; see Example~\ref{ex1a}.}
\end{Example}

The next two-dimensional example presents a simple linear-quadratic variational system of type \eqref{VS} with $\thy$ from \eqref{theta} such that a stationary point therein is associated with both critical and noncritical Lagrange multipliers.

\begin{Example}{\bf(variational systems with both critical and noncritical multipliers corresponding to a given stationary point).}\label{ex2} {\rm Specify the data of \eqref{theta} and \eqref{VS} as follows:
\begin{equation}\label{kd01}
Y:=\R^2_+,\quad\dps B:=\begin{pmatrix}1&0\\0&0\end{pmatrix},\quad f(x):=-x,\;\mbox{ and }\;\Phi(x):=(0,x^2)\;\mbox{ for }\;x\in\R.
\end{equation}
Thus we have in \eqref{VS} that $\Psi(x,\lm)=f(x)+\nabla\Phi(x)^*\lm=-x+2x\lm_2$ for any $x\in\R$ and $\lm=(\lm_1,\lm_2)\in\R^2$.
By Theorem~\ref{gdt}(i), we obtain $ \dom\thy=\R \times \R_-$.
 Since $\p\thy(u)=(N_Y+B)^{-1}(u)$ by Theorem~\ref{gdt}(iii), it is not hard to see $\p\thy(0)=\{0\}\times\R_+$, and so $\Lambda(\ox)=\{0\}\times\R_+$ with $\ox:=0$.
 Then for any $\lm=(\lm_1,\lm_2)\in\Lambda(\ox)$ we get $\lm_1=0$ and $\lm_2\ge 0$. On the other hand, conditions \eqref{crit} from Theorem~\ref{desc}  read now as
\begin{equation*}
(2\lm_2-1)\xi= 0,\;\;\inp{-B\eta}{\eta}=0,\;\;-B\eta\in\ok^*,\;\;\eta\in\ok.
\end{equation*}
This tells us that if $\lm_2\ne\frac{1}{2}$, the latter system admits only the solution $\xi=0$, and thus the obtained Lagrange multiplier $\lm$ is noncritical. In the case where $\lm_2=\frac{1}{2}$, this system admits nontrivial solutions $\xi$, and so the Lagrange multiplier $\lm=(0,\frac{1}{2})$ is critical.}
\end{Example}

\section{Uniqueness of Lagrange Multipliers and Isolated Calmness}\label{unique}

This section is devoted to the study of uniqueness of Lagrange multipliers corresponding to given stationary points of the variational systems \eqref{VS} with piecewise linear-quadratic penalties \eqref{theta}. This issue is definitely of its own interest while seems to be independent of multiplier criticality. However, the methods we develop for the uniqueness study and the obtained conditions for it occur to be closely related to the subsequent characterizations of noncritical multiplies as well as their deeper understanding and specification.

First we recall some ``at-point" (vs.\ ``around/neighborhood") stability properties of set-valued mappings that have been recognized in variational  analysis; see, e.g., \cite{dr,m18,rw} with the references and commentaries therein.

It is said that a mapping $F\colon\R^n\rightrightarrows\R^m$ is {\em calm} at $(\ox,\oy)\in\gph F$ if there exist a constant $\ell\ge 0$ and neighborhoods $U$ of $\xb$ and $V$ of $\yb$ such that
\begin{equation}\label{calm}
F(x)\cap V\subset F(\xb)+\ell\|x-\xb\|\B\;\textrm{ for all }\;x\in U,
\end{equation}
where $\B$ stands for the closed unit ball of the space in question. If \eqref{calm} is replaced by
\begin{equation}\label{iso-calm}
F(x)\cap V \subset\big\{\yb\big\}+\ell\n{x-\xb}\B\;\textrm{ for all }\;x\in U,
\end{equation}
then the corresponding property is known as {\em isolated calmness} of $F$ at $(\ox,\oy)$. If the $\gph F$ is locally closed at $(\ox,\oy)$, the latter property admits the graphical derivative characterization
\begin{equation}\label{grdcr}
DF(\xb,\yb)(0)=\{0\}
\end{equation}
known as the {\em Levy-Rockafellar criterion}; see the commentaries to \cite[Theorem~4E.1]{dr}.

Finally, $F$ enjoys the {\em robust isolated calmness}  property at $(\ox,\oy)$ if in addition to \eqref{iso-calm} we have $F(x)\cap V\ne\emp$. This name is coined quite recently \cite{dsz}, while the property itself has been actually used in optimization over the years; see the discussions in \cite{dsz,ms17}.\vspace*{0.05in}

In this section we employ the calmness and isolated calmness properties for characterizations of uniqueness of Lagrange multipliers in \eqref{VS} with the piecewise linear-quadratic term \eqref{theta}. Robust isolated calmness is used in the last section of the paper.

Using the data of \eqref{VS}, consider the set-valued mapping $G\colon\R^n\times\R^m\rightrightarrows\R^n\times\R^m$ given by
\begin{equation}\label{g}
G(x,\lm):=\begin{pmatrix}
\Psi(x,\lm)\\
-\Phi(x)\end{pmatrix}+\begin{pmatrix}
0\\
(\partial\thy)^{-1}(\lm)
\end{pmatrix}\;\mbox{ for  all }\;(x,\lm)\in\R^n\times\R^m.
\end{equation}
Then fix a point $\ox\in \R^n$ and define the parameterized {\em multiplier map} $M_{\xb}\colon\Rn\times\Rm\rightrightarrows\Rm$ associated with $\xb$ by
\begin{equation} \label{lmm}
M_{\xb}(p_1,p_2):=\big\{\lm\in\Rm\big|\;(p_2,p_2)\in G(\xb,\lm)\big\},\quad(p_1,p_2)\in\R^n\times\Rm.
\end{equation}
We have $M_{\xb}(0,0)=\Lambda(\xb)$ for the Lagrange multiplier set \eqref{lag} of the unperturbed system \eqref{VS}.

The next theorem characterizes uniqueness of Lagrange multipliers in variational systems \eqref{VS} with the term $\thy$ from \eqref{theta} via both calmness and isolated calmness properties of the multiplier map \eqref{lmm}, which are equivalent to each other in this case and are characterized in turn by a novel {\em dual qualification condition}.

\begin{Th}{\bf(characterizations of uniqueness of Lagrange multipliers in variational systems).}\label{uniq}
Let $(\bar{x},\bar{\lm})$ be a solution to the variational system  \eqref{VS} with $\thy$ taken from \eqref{theta}.
Then the following properties are equivalent:\vspace{-0.2 cm}
\begin{itemize}[noitemsep]
\item[\bf(i)] $\Lambda(\ox)=\{\olm\}$.
\item[\bf(ii)]$M_{\xb}$ is calm at $\big((0,0),\olm\big)$ and $\Lambda(\xb)=\{\olm\}$.
\item[\bf(iii)] $M_{\xb}$ is isolatedly calm at $\big((0,0),\olm\big)$.
\item[\bf(iv)]We have the dual qualification condition
\begin{equation}\label{dqc}
D\p\thy(\zb,\olm)(0)\cap \ker\nabla\Phi(\xb)^*=\{0\},
\end{equation}
where $D\p\thy(\zb,\olm)$ is calculated by \eqref{desc1}.
\end{itemize}
\end{Th}
\begin{proof}
Denoting $\oz:=\Phi(\ox)$ as above, we begin with proving the equivalence (iii)$\Longleftrightarrow$(iv). To proceed, observe that the graph of $M_{\ox}$ is closed and deduce from \eqref{grdcr} that $M_{\ox}$ is isolatedly calm at $((0,0),\olm)$  if and only if
$DM_{\ox}\big((0,0),\olm\big )(0,0)=\{0\}$. It is not hard to check that $\eta\in DM_{\ox}\big((0,0),\olm\big )(0,0)$ amounts to saying that
$\eta$ is a solution to the system
\begin{equation*}
\left[\begin{array}{c}
0\\
0
\end{array}
\right]\in\left[\begin{array}{c}
\nabla\Phi(\ox)^*\eta\\
0
\end{array}
\right]+\left[\begin{array}{c}
0\\
D(\thy)^{-1}(\olm,\oz)(\eta)
\end{array}
\right].
\end{equation*}
This tells us that $\eta$ is a solution to the above system if and only if
\begin{equation*}
\eta\in D\p\thy(\zb,\olm)(0)\cap\ker\nabla\Phi(\xb)^*.
\end{equation*}
Combining these facts verifies the equivalence between conditions (iii) and (iv).

Next we show that (i)$\Longrightarrow$(iv). Assume on the contrary that the dual qualification condition \eqref{dqc} fails while (i) holds, and so  find an element
\begin{equation*}
\eta\in D\p\thy(\zb,\olm)(0)\cap\ker\nabla\Phi(\xb)^*\;\mbox{ such that }\;\eta\ne 0.
\end{equation*}
Since $\Psi(\xb,\olm+t\eta)=0$ for any $t>0$, we get from $\eta\in D\p\thy(\zb,\olm)(0)$ and \eqref{gdr} that $\eta\in\p\thk(0)$, and hence $-B\eta\in N_{\ok}(\eta)$ by Theorem~\ref{gdt}(ii). Choosing $t$ to be sufficiently small and employing the Reduction Lemma from \cite[Lemma~2E.4]{dr} ensure the existence of a neighbored $U$ of $(0,0)\in \R^m\times \R^m$ such that
$$
t(\eta, -B\eta)\in [\gph N_\ok]\cap U=\big[\gph N_Y- (\bar\lm,\oz-B\olm)\big]\cap U.
$$
This in turn results in  $\zb-B\olm-tB\eta\in N_Y(\olm+t\eta)$, which yields by \eqref{fo} the inclusion $\olm+t\eta\in\p\thy(\zb)$. Combining the latter with $\Psi(\xb,\olm+t\eta)=0$ results in $\olm+t\eta\in\Lambda(\xb)$. However, we have $\eta\ne 0$ thus $\olm+t\eta\ne\olm$ for any $t>0$, which contradicts (i) and so verifies the claimed implication (i)$\Longrightarrow$(iv).

To show further that the isolated calmness of $M_{\ox}$ at $\big((0,0),\olm\big)$ imposed in (iii) yields (ii), it suffices to check that $\Lambda(\xb)=\{\olm\}$.  Indeed, the assumed isolated calmness allows us to find a neighborhood $O$ of $\bar \lm$ such that $M_{\ox}(0,0)\cap O=\{\olm\}$, which tells us by the convex-valuedness of $M_{\ox}$ that $M_{\ox}(0,0)=\{\olm\}$. Combining the latter with $M_{\xb}(0,0)=\Lambda(\xb)$ verifies (ii). Since (ii) obviously implies (i), we complete the proof of the theorem.
\end{proof}\vspace*{0.05in}

The next example reveals that the dual qualification condition \eqref{dqc} is essential for the uniqueness of Lagrange multipliers in Theorem~\ref{uniq}.

\begin{Example}{\bf(nonuniqueness of Lagrange multipliers under failure of the dual qualification condition).}\label{manymul} {\rm Consider the variational system \eqref{VS} with term \eqref{theta}, where $Y$ and $B$ are taken from \eqref{kd01}, while $\Phi\colon\R^2\to\R^2$ is defined by  $\Phi(x_1,x_2):=(x_1,0)$ and $f\colon\R^2\to\R^2$ is defined by $f(x)=0$ for all $x\in\R^2$. It is shown in Example~\ref{ex2} that $\dom\thy=\R\times \R_-$. Letting $\xb:=(0,0)$, we get by the direct calculation that
\begin{equation*}
\p\thy(\xb)=\{0\}\times\R_+\;\mbox{ and }\;\Psi(\ox,\lm)=\nabla\Phi(x)^*\lm=(\lm_1,0),
\end{equation*}
and so $\Lambda(\xb)=\{0\}\times\R_+$, which is not a singleton.

Let us now show that the dual qualification condition fails in this setting. Having $\ker\nabla\Phi(\xb)^*=\{0\}\times\R$ and choosing $\olm:=(0,0)$ give us the critical cone
\begin{equation*}
\ok=T_Y(\olm)\cap\big\{\Phi(\ox)-B\olm\big\}^\perp=Y,
\end{equation*}
and so $\p\thk(0,0)=\{0\}\times\R_+$. Combining it with \eqref{gdr}, we arrive at
\begin{equation*}
\p\thk(0,0)\cap\ker\nabla\Phi(\xb)^*=D\p\thy(\zb,\olm)(0,0)\cap\ker\nabla\Phi(\xb)^*=\{0\}\times\R_+\ne\{(0,0)\},
\end{equation*}
which demonstrates the failure of the dual qualification condition \eqref{dqc}.}
\end{Example}

\section{Characterizations of Noncritical Multipliers}\label{noncrit-char}

In this section we derive major characterizations of noncritical multipliers for the piecewise linear-quadratic variational systems \eqref{VS} in terms of semi-isolated calmness and error bounds.

Using the mapping $G$ from \eqref{g}, define the {\em solution map} $S\colon\mathbb{R}^n\times\mathbb{R}^m\rightrightarrows\mathbb{R}^n\times\mathbb{R}^m$ for the {\em canonical perturbation} of system \eqref{VS} by
\begin{equation}\label{s}
S(p_1,p_2):=\big\{(x,\lm)\in\mathbb{R}^n\times\mathbb{R}^m\big|\;(p_1,p_2)\in G(x,\lm)\big\}.
\end{equation}

The property of {\em semi-isolated calmness} used in \eqref{upper} was introduced in \cite{ms17} for solution maps to general variational systems with a product structure of values as in \eqref{s}. The reader can see that for such mappings the semi-isolated calmness of the variational systems of type \eqref{VS} occupies an intermediate position between the calmness and isolated calmness.

In what follows we use the notation $\dist(x;\O)$ for the distance between a point $x\in\R^n$ and a set $\O\subset\R^n$, $\B_\ve(x)$ for the closed ball centered at $x\in\R^n$ with radius $\ve>0$, and
\begin{equation}\label{pr}
P\ph(x):={\rm argmin}\Big\{\ph(u)+\frac{1}{2}\|x-u\|^2\Big|\;u\in\R^n\Big\},\quad x\in\R^n,
\end{equation}
for the {\em proximal mapping} $P\ph\colon\R^n\tto\R^n$ associated with a function $\ph\colon\R^n\to\oR$.

\begin{Th}{\bf(major characterizations of noncritical multipliers in variational systems).}\label{charc} Let $(\bar{x},\bar{\lm})$ be a solution to the variational system \eqref{VS} with the piecewise linear-quadratic term \eqref{theta}. Then the following conditions are equivalent:
\begin{itemize}[noitemsep]
\item[\bf(i)] The Lagrange multiplier $\bar{\lm}$ is noncritical for \eqref{VS} corresponding to $\ox$.
\item[\bf(ii)]
There exist numbers $\ve>0$, $\ell\ge 0$ and neighborhoods $U$ of $0\in\R^n$ and $W$ of $0\in\R^m$ such that for any $(p_1,p_2)\in U\times W$ the following inclusion holds:
\begin{equation}\label{upper}
S(p_1,p_2)\cap\B_\ve(\ox,\olm)\subset\big[\{\ox\}\times\Lambda(\ox)\big]+\ell\big(\|p_1\|+\|p_2\|\big)\B.
\end{equation}
\item[\bf(iii)] There exist numbers $\ve>0$ and $\ell\ge 0$ such that the error bound estimate
\begin{equation*}\label{subr}
\|x-\ox\|+\dist\big(\lm;\Lambda(\ox)\big)\le\ell\big(\|\Psi(x,\lm)\|+\dist\big(\Phi(x);(\partial\thy)^{-1}(\lm)\big)\big)
\end{equation*}
holds for any $(x,\lm)\in\B_\ve(\ox,\olm)$ in terms of the inverse subdifferential of $\thy$.
\item[\bf(iv)]  There are numbers $\ve>0$ and $\ell\ge 0$ such that the error bound estimate
\begin{equation}\label{subr3}
\|x-\ox\|+\dist\big(\lm;\Lambda(\ox)\big)\le\ell\big(\|\Psi(x,\lm)\|+\|\Phi(x)-(P\thy)(\lm+\Phi(x))\|\big)
\end{equation}
holds for any $(x,\lm)\in\B_\ve(\ox,\olm)$ in terms of the proximal mapping $P\thy$ from \eqref{pr}.
\end{itemize}
\end{Th}
\begin{proof} Let us first verify that (ii) implies (i). Theorem~\ref{desc} reduces it to proving that the semi-isolated calmness property in (ii) ensures that for any solution $(\xi,\eta)\in\mathbb{R}^n\times\mathbb{R}^m$ to the system \eqref{cri} we have $\xi=0$. Define $(x_t,\lm_t):=(\bar{x}+t\xi,\bar{\lm}+t\eta)$ for all $t>0$ and observe that
\begin{equation*}
\begin{array}{ll}
\Psi(x_t,\lm_t)-\Psi(\bar{x},\bar{\lm})&=\big(f(x_t)-f(\bar{x})\big)+\big(\nabla\Phi(x_t)-\nabla\Phi(\bar{x})\big)^*\bar{\lm}+t\nabla\Phi(x_t)^*\eta \\
&=t\nabla f(\bar{x})\xi+o(t)+t\big(\nabla^2\Phi(\bar{x})\xi\big)^*\bar{\lm}+t\nabla\Phi(\bar{x})^*\eta+o(t)\\
&=t\big(\nabla_x\Psi(\bar{x},\bar{\lm})\xi+\nabla\Phi(\bar{x})^*\eta\big)+o(t)=o(t)
\end{array}
\end{equation*}
whenever $t$ is sufficiently small. Letting $p_{1t}:=\Psi(x_t,\lm_t)$ and using $\Psi(\bar{x},\bar{\lm})=0$, we deduce from the last equality above that $p_{1t}=o(t)$. It follows in the similar way that
\begin{equation*}
\Phi(x_t)=\Phi(\bar{x})+t\nabla\Phi(\bar{x})\xi+o(t)\;\mbox{ for all small }\;t>0.
\end{equation*}
Denoting further $z_t:=\Phi(\bar{x})+t\nabla\Phi(\bar{x})\xi$ implies that
\begin{equation*}
z_t-\Phi(x_t)=o(t)\;\mbox{ as }\;t>0,
\end{equation*}
and therefore we get $p_{2t}=o(t)$ for $p_{2t}:=z_t-\Phi(x_t)$.

Let us now prove that $(x_t,\lm_t)\in S(p_{1t},p_{2t})$ for $t>0$ sufficiently small. Since $p_{1t}=\Psi(x_t,\lm_t)$, we only need to verify by Theorem~\ref{gdt}(ii) that
\begin{equation}\label{lm-th}
\lm_t\in\partial\thy(z_t)=(N_Y+B)^{-1}(z_t),\;\mbox{ or equivalently }\;z_t-B\lm_t\in N_Y(\lm_t).
\end{equation}
To proceed with checking \eqref{lm-th}, deduce from \eqref{cri} that
\begin{equation*}
\eta\in\ok=K_Y(\olm,\zb-B\olm)=T_Y(\bar{v})\cap\{\bar{z}-B\bar{\lm}\}^\perp.
\end{equation*}
Denoting $\lm_t:=\olm+t\eta$ and remembering that $Y$ is a convex polyhedral set, we conclude that $\lm_t\in Y$ for all $t>0$ sufficiently small. Furthermore, it follows from \eqref{cri} that
\begin{equation*}
\nabla\Phi(\bar{x})\xi-B\eta\in\ok^*=N_Y(\olm)+\R(\bar{z}-B\bar{\lm}).
\end{equation*}
Thus there exist $\alpha\in\mathbb{R}$ and $w\in N_Y(\bar{\lm})$ such that $\nabla\Phi(\bar{x})\xi-B\eta=\alpha(\bar{z}-B\bar{\lm})+w$. Using this together with \eqref{cri} gives us the equalities
\begin{equation*}
0=\inp{\nabla\Phi(\bar{x})\xi-B\eta}{\eta}=\alpha\inp{\bar{z}-B\bar{\lm}}{\eta}+\inp{w}{\eta}=\inp{w}{\eta}.
\end{equation*}
Recall that $ N_Y(\bar{\lm})=\{\sum_{i\in I(\bar{\lm})}\beta_i b_i|\;\beta_i\ge 0\}$, where $I(\olm)$ stands for the set of active constraints in $Y$ at $\olm$. It allows us to deduce from the inclusion $w\in N_Y(\bar{\lm})$ that there are numbers $\beta_i\ge 0$ as $i\in I(\bar{\lm})$ such that $w=\sum_{i\in I(\bar{\lm})}\beta_i b_i$, and therefore
\begin{equation*}
\sum_{i\in I(\bar{\lm})}\beta_i\inp{b_i}{\eta}=\inp{w}{\eta}=0.
\end{equation*}
Observe furthermore the relationships
\begin{equation*}
z_t-B\lm_t=\Phi(\bar{x})+t\nabla\Phi(\bar{x})\xi-B\bar{\lm}-tB\eta=\bar{z}-B\bar{\lm}+t(\nabla\Phi(\bar{x})\xi-B\eta)=(1+t\alpha)(\bar{z}-B\bar{\lm})+tw,
\end{equation*}
where $1+t\alpha>0$ for small $t>0$. Since both $\bar{z}-B\bar{\lm}$ and $w$ belong to $N_Y(\bar{\lm})$, it follows that  $(1+t\alpha)(\bar{z}-B\bar{\lm})+tw\in N_Y(\bar{\lm})$, and thus there is $\tau_{it}\ge 0$ for $i\in I(\bar{\lm})$ such that $z_t-B\lm_t=\sum_{i\in I(\bar{\lm})}\tau_{it}b_i$. Noting that $\inp{z_t-B\lambda_t}{\eta} = 0$ and $\inp{b_i}{\eta} \leq 0$ for all $i \in I(\olm)$, we deduce that
\begin{equation}\label{act}
\inp{b_i}{\eta} = 0 \textrm{ for all } i \in I(\olm) \textrm{ with } \tau_{it}>0.
\end{equation}
Let us now show that
\begin{equation*}
\tau_{it}=0\;\mbox{ if }\;i\in I(\bar{\lm})\setminus I(\lm_t).
\end{equation*}
Suppose on the contrary that there is an index $i_0\in I(\bar{\lm})\setminus I(\lm_t)$ for which $\tau_{i_0t}>0$. This means that $\inp{b_{i_0}}{\bar{\lm}}=\alpha_{i_0}$ and $\inp{b_{i_0}}{\lm_t}<\alpha_{i_0}$. Therefore
\begin{equation*}
\inp{b_{i_0}}{\bar{\lm}}+t\inp{b_{i_0}}{\eta}=\inp{b_{i_0}}{\lm_t}<\alpha_{i_0},
\end{equation*}
which in turn yields $\inp{b_{i_0}}{\eta}<0$, a contradiction with \eqref{act}. Thus for all $i\in I(\bar{\lm})\setminus I(\lm_t)$ we get $\tau_{it}=0$ and hence arrive at
\begin{equation*}
z_t-B\lm_t=\sum_{i\in I(\lm_t)}\tau_{it}b_i\in N_Y(\lm_t).
\end{equation*}
This verifies \eqref{lm-th} and thus implies that $(x_t,\lm_t)\in S(p_{1t},p_{2t})$. It now follows from the assumed semi-isolated calmness \eqref{upper} in (ii) that
\begin{equation*}
\n{\xi}=\frac{\n{x_t-\bar{x}}}{t}\le\frac{\ell\big(\n{p_{1t}}+\n{p_{2t}}\big)}{t},
\end{equation*}
which results in $\xi=0$ by letting $t\dn 0$. It tells us $\bar{\lm}$ is noncritical and hence justify the implication (ii)$\implies$(i) of the theorem.\vspace*{0.05in}

Next we prove the opposite implication (i)$\Longrightarrow$(ii). Assuming that the multiplier noncriticality in (i) holds, let us first verify the following statement.\\[1ex]
{\bf Claim:} {\em There exist numbers $\ve>0$ and $\ell\ge 0$ and neighborhoods $U$ of $0\in\R^n$ and $W$ of $0\in \R^m$ such that
for any $(p_1,p_2)\in U\times W$ and $(x_{p_1p_2},\lm_{p_1p_2})\in S(p_1,p_2)\cap\mathbb{B}_\ve(\bar{x},\bar{\lm})$ we have
\begin{equation} \label{up1}
\n{x_{p_1p_2}-\bar{x}}\le\ell\big(\n{p_1}+\n{p_2}\big).
\end{equation}}\vspace*{-0.15in}

To justify this claim, suppose on the contrary that \eqref{up1} fails and thus for any $k\in\mathbb{N}$ find $(p_{1k},p_{2k})\in\mathbb{B}_{1/k}(0)\times\mathbb{B}_{1/k}(0)$, $k\in\mathbb{N}$, and $(x_k,\lm_k)\in S(p_{1k},p_{2k})\cap \mathbb{B}_{1/k}(\bar{x},\bar{\lm})$ such that
\begin{equation*}
\frac{\n{p_{1k}}+\n{p_{2k}}}{\n{x_k-\bar{x}}}\rightarrow 0\;\mbox{ as }\;k\to\infty.
\end{equation*}
Denote $t_k:=\|x_k-\ox\|$ and deduce from the convergence above that $p_{1k}=o(t_k)$ and $p_{2k}=o(t_k)$. Since $\thy$ is a convex piecewise linear-quadratic function, it follows from the proof of \cite[Theorem~11.14(b)]{rw} that $\gph\partial\thy$ is a union of finitely many convex  polyhedral sets. This together with \cite[Theorem~3D.1]{dr} and $\oz:=\Phi(\ox)\in\dom\partial\thy$ ensures the existence of a number $\ell'\ge 0$ and a neighborhood $O$ of $\oz$ such that for all $z\in O\cap\dom\partial\thy$ we have
\begin{equation}\label{ulp}
\partial\thy(z)\subset\partial\thy(\zb)+\ell'\n{z-\zb}\mathbb{B}.
\end{equation}
Suppose without loss of generality that $z_{k}:=p_{2k}+\Phi(x_{k})\in O$ for all $k\in\mathbb{N}$. Since $\lm_k\in\partial\thy(z_k)$, there exist
 $\lm\in\partial\thy(\oz)$ and $b\in\B$ such that $\lm_k=\lm+\ell'\n{z_k-\zb}b$. Using this along with the classical Hoffman lemma, we find a number $M\ge 0$ such that
\begin{eqnarray}\label{err}
\begin{array}{ll}
\mathrm{dist}\big(\lm_{k};\Lambda(\bar{x})\big)&\le M\left(\n{\Psi(\bar{x},\lm_{k})}+\dist\big(\lm_k;\partial\thy(\oz)\big)\right)\\
&\le M\n{\Psi(\bar{x},\lm_{k})-\Psi(x_{k},\lm_{k})}+M \n{\Psi(x_{k},\lm_{k})}+\ell'\n{z_{k}-\zb}\\
&\le M\rho(1+\|\lm_k\|)\n{x_{k}-\bar{x}}+M\n{p_{1k}}+\ell'\rho\n{x_{k}-\ox}+\ell'\|p_{2k}\|,
\end{array}
\end{eqnarray}
where $\rho$ is a common calmness constant for the mappings $f$, $\Phi$, and $\nabla\Phi$ at $\ox$. Since $\Lambda(\bar{x})$ is closed and convex, for each $k\in\N$ there exists a vector $\mu_k\in\Lambda(\bar{x})$ for which
\begin{equation*}
\frac{\|\lm_k-\mu_k\|}{t_k}\le M\rho(1+\|\lm_k\|)+M\frac{\n{p_{1k}}}{t_k}+\ell'\rho+\ell'\frac{\|p_{2k}\|}{t_k},\quad k\in\N.
\end{equation*}
Thus we can assume without loss of generality that
\begin{equation*}
\frac{\lm_k-\mu_k}{t_k}\rightarrow\tilde{\eta}\;\mbox{ for some }\;\tilde\eta\in\R^m.
\end{equation*}
By passing to a subsequence if necessary, it follows that
\begin{equation*}\label{xi}
\frac{x_k-\ox}{t_k}\to\xi\;\mbox{ as }\;k\to\infty\;\mbox{ with some }\;0\ne\xi\in\R^n.
\end{equation*}
Due to $\mu_k\in\Lambda(\bar{x})$ and the discussions above we get the equalities
\begin{equation*}
\begin{array}{ll}
o(t_k)=p_{1k}&=\Psi(x_k,\mu_k)=\Psi(x_k,\mu_k)-\Psi(\bar{x},\mu_k)+\nabla\Phi(x_k)^*(\lm_k-\mu_k)\\
&=\nabla_x\Psi(\bar{x},\mu_k)(x_k-\bar{x})+\nabla\Phi(x_k)^*(\lm_k-\mu_k)+o(t_k),
\end{array}
\end{equation*}
which lead us as $k\to\infty$ to the limiting condition
\begin{equation} \label{pl03}
\nabla_x\Psi(\bar{x},\bar{\lm})\xi+\nabla\Phi(\bar{x})^*\tilde{\eta}=0,
\end{equation}
It further follows from $(x_k,\lm_k)\in S(p_{1k},p_{2k})$ that $\lm_k\in\partial\thy(z_k)$, which is equivalent to the inclusion $z_k-B\lm_k\in N_Y(\lm_k)$ for each $k\in\N$ by Theorem~\ref{gdt}(ii). Since $Y$ is a convex polyhedral set, the Reduction Lemma from \cite[Lemma~2E.4]{dr}) tells us that
\begin{equation*}
z_k-B\lm_k-(\zb-B\olm)\in N_\ok(\lm_k-\olm)
\end{equation*}
for all $k\in\N$ sufficiently large, where $\ok$ is the critical cone to $Y$ at $\oz$ for $\oz-B\olm$ taken from Theorem~\ref{gdt}(iii).
This along with Theorem~\ref{gdt}(iii) brings us to the conclusions
\begin{equation*}
\lm_k-\olm\in\partial\thk(z_k-\zb)=D\partial\thy(\bar{z},\bar{\lm})(z_k-\oz),\;\mbox{ and so}
\end{equation*}
\begin{equation}\label{pl01}
\frac{\lm_k-\olm}{t_k}\in D\partial\thy(\bar{z},\bar{\lm})\Big(\frac{z_k-\oz}{t_k}\Big)=\partial\thk\Big(\frac{z_k-\zb}{t_k}\Big),
\end{equation}
which imply in turn that $\dps\frac{z_k-\zb}{t_k}\in\dom\partial\thk$. Since $\ok$ is a convex polyhedral set, it follows from Theorem~\ref{gdt}(i) that $\thk$ is a convex piecewise linear-quadratic function. Thus \cite[Proposition~10.21]{rw} tells us that $\dom\partial\thk=\dom\thk$. Employing Theorem~\ref{gdt}(i) ensures that $\dom\thk$ is a closed set. Combining it with the convergence $\dps \frac{z_k-\zb}{t_k}\rightarrow\nabla\Phi(\xb)\xi$ as $k\rightarrow\infty$ yields
\begin{equation}\label{pl02}
\nabla\Phi(\xb)\xi\in\dom\partial\thk.
\end{equation}
Since $\mu_k\in\Lambda(\bar{x})$, we get $\mu_k\in\partial\thy(\oz)$ and, proceeding similarly to the proof of \eqref{pl01}, arrive at
\begin{equation*}
\frac{\mu_k-\olm}{t_k}\in\partial\thk(0).
\end{equation*}
Furthermore, it follows from $\olm\in\Lambda(\xb)$ and $\mu_k\in\Lambda(\xb)$ that $\olm-\mu_k\in\ker\nabla\Phi(\xb)^*$. Using  \eqref{pl02} and arguing as in the proof of
 \eqref{ulp}, we find $\ell'\ge 0$ and a neighborhood $O$ of $\nabla\Phi(\xb)\xi$ such that
\begin{equation*}
\partial\thk(u)\subset\partial\thk\big(\nabla\Phi(\xb)\xi\big)+\ell'\n{u-\nabla\Phi(\xb)\xi}\mathbb{B}
\end{equation*}
for all $u\in O\cap\dom\partial\thk$. Employing the latter together with \eqref{pl01} leads us to the relationships
\begin{eqnarray*}
\frac{\lm_k-\mu_k}{t_k}&=&\frac{\lm_k-\olm}{t_k}+\frac{\olm-\mu_k}{t_k}\\
&\in&\partial\thk\Big(\frac{z_k-\zb}{t_k}\Big)-\big[\mathrm{ker}\nabla\Phi(\xb)^*\cap\partial\thk(0)\big]\\
&\subset&\partial\thk(\nabla\Phi(\xb)\xi)+\ell'\big\|\frac{z_k-\zb}{t_k}-\nabla\Phi(\xb)\xi\big\|\B-\big[\ker\nabla\Phi(\xb)^*\cap\partial\thk(0)\big].
\end{eqnarray*}
This allows us to find, for all $k\in\N$ sufficiently large, a $b_k\in\B$ such  that
\begin{equation}\label{inc1}
\frac{\lm_k-\mu_k}{t_k}-\ell'\big\|\frac{z_k-\zb}{t_k}-\nabla\Phi(\xb)\xi\big\|b_k\in\partial\thk(\nabla\Phi(\xb)\xi)-\big[\ker\nabla\Phi(\xb)^*\cap \partial\thk(0)\big].
\end{equation}
We can see that the left-hand side of inclusion \eqref{inc1} converges as $k\to\infty$ to the vector $\tilde{\eta}$. On the other hand, the right-hand side of this inclusion is the sum of two convex polyhedral sets, and so  is closed. This shows that $\tilde\eta$ satisfies to
\begin{equation}\label{u67}
\tilde{\eta}\in\partial\thk(\nabla\Phi(\xb)\xi)-\big[\ker\nabla\Phi(\xb)^*\cap\partial\thk(0)\big].
\end{equation}
Thus  we get vectors $\eta\in\partial\thk(\nabla\Phi(\xb)\xi)$ and $\eta'\in \ker\nabla\Phi(\xb)^*\cap\partial\thk(0)$, which provide the representation $\tilde{\eta}=\eta-\eta'$. It follows from the relationship \eqref{gdr} in Theorem~\ref{gdt}(iii) that $\eta\in D\partial\thy(\zb,\olm)(\nabla\Phi(\xb)\xi)$. Furthermore, employing \eqref{pl03} tells us that
\begin{equation*}
0=\nabla_x\Psi(\bar{x},\bar{\lm})\xi+\nabla\Phi(\bar{x})^*\tilde\eta=\nabla_x\Psi(\bar{x},\bar{v})\xi+\nabla\Phi(\bar{x})^*\eta,
\end{equation*}
which contradicts the noncriticality of $\olm$ due to $\xi\ne 0$ and thus completes the proof of the claim.\vspace*{0.05in}

To finalize verifying implication (i)$\Longrightarrow$(ii) in the theorem, take the neighborhoods $U$ and $W$ from the above claim and shrink them if necessary for the subsequent procedure. Using the claim and arguing similarly to the proof of the conditions in \eqref{err} give us
a constant $\ell'\ge 0$ such that for any $(p_1,p_2)\in U\times W$ and any $(x_{p_1p_2},\lm_{p_1p_2})\in S(p_1,p_2)\cap\mathbb{B}_\ve(\bar{x},\bar{\lm})$ we have
\begin{equation}\label{pap}
\dist\big(\lm_{p_1p_2};\Lambda(\bar{x})\big)\le\ell'\big(\n{x_{p_1p_2}-\bar{x}}+\n{p_1}+\n{p_2}\big).
\end{equation}
Combining it with \eqref{up1} allows us to find $\ell\ge 0$ for which $(p_1,p_2)\in U\times W$ and
\begin{equation*}
\n{x_{p_1p_2}-\bar{x}}+\dist\big(\lm_{p_1p_2};\Lambda(\bar{x})\big)\le\ell\big(\n{p_1}+\n{p_2}\big)
\end{equation*}
whenever $(x_{p_1p_2},\lm_{p_1p_2})\in S(p_1,p_2)\cap\mathbb{B}_\ve(\bar{x},\bar{\lm})$. This clearly justifies the semi-isolated calmness property \eqref{upper} and thus finishes the proof of implication (i)$\implies$(ii).

The equivalence between (ii) and (iii) can be verified similarly to the corresponding arguments in the proof of \cite[Theorem~4.1]{ms17}, and so we omit them here. Thus it remains to establish the equivalence between assertions (ii) and (iv) of the theorem to complete its proof.
	
Let us start with checking implication (iv)$\Longrightarrow$(ii). Picking $(p_1,p_2)\in\B_{\ve}(0,0)$ and $(x,\lm)\in S(p_1,p_2)\cap\B_{\ve}(\ox,\olm)$ with $\ve$ and $\ell$ taken from (iv), we get from the definition of $S$ that
\begin{equation}\label{inc2}
\Psi(x,\lm)=p_1\;\mbox{ and }\;\lm\in\p\thy(\Phi(x)+p_2).
\end{equation}
It follows from \cite[Proposition~12.19]{rw} due to the convexity of $\thy$ that $P\thy=(I+\p\thy)^{-1}$, and hence the second inclusion in \eqref{inc2} is equivalent to the equality $P\thy(\lm+\Phi(x)+p_2)=\Phi(x)+p_2$. Appealing now to \eqref{subr3} brings us to the estimates
\begin{eqnarray*}
\|x-\ox\|+\dist\big(\lm,\Lm(\ox)\big)&\le&\ell\big(\|\Psi(x,\lm)\|+\|\Phi(x)-P\thy(\lm +\Phi(x)\|\big)\\
&\le&\ell\big(\|p_1\|+\|P\thy(\lm+\Phi(x)+p_2)-P\thy(\lm +\Phi(x))\| +\|p_2\|\big)\\
&\le&\ell\big(\|p_1\|+\|p_2\|+\|p_2\|\big),
\end{eqnarray*}
which readily justify the assertion in (ii).

Finally, we verify the converse implication (ii)$\Longrightarrow$(iv). To proceed, pick $(x,\lm)\in\B_{\ve/2}(\ox,\olm)$, where $\ve$ is taken from (ii). Define the vectors
\begin{equation}\label{ee1}
p_2:= P\thy\big(\lm+\Phi(x)\big)-\Phi(x)\;\mbox{ and }\;p_1:=\Psi(x,\lm-p_2).
\end{equation}
Since $\Phi$ and  $\nabla \Phi$ are continuous at $\ox$ and since $P\thy$ is Lipschitz continuous,
we assume without loss of generality that $(p_1,p_2)\in\B_{\ve/2}(0,0)$ and $\B_{\ve/2}(0,0)\subset U\times W$, where $U$ and $W$ come from (ii).
It follows from \eqref{ee1} that $(x,\lm-p_2)\in S(p_1,p_2)\cap\B_{\ve}(\ox,\olm)$. Since $\nabla \Phi$ is continuous at $\ox$, we can assume without loss
generality that for some $\rho>0$ we have $\|\nabla \Phi(x)\|\leq \rho$ for all $x\in\B_\ve(\ox)$.
So we deduce from \eqref{upper} that
\begin{eqnarray*}
\|x-\ox\|+\dist\big(\lm-p_2,\Lm(\ox)\big)&\le&\ell\big(\|p_1\|+\|p_2\|\big)\\
&\le&\ell\big(\|\Psi(x,\lm-p_2)\|+\|\Phi(x)-P\thy(\lm +\Phi(x))\|\big)\\
&\le&\ell\big(\|\Psi(x,\lm)\|+\rho\|p_2\|+\|\Phi(x)-P\thy(\lm +\Phi(x))\|\big)\\
&\le&\ell\big(\|\Psi(x,\lm)\|+(\rho+1)\|\Phi(x)-P\thy(\lm+\Phi(x))\|\big).
\end{eqnarray*}
Recall that the distance function $\dist\big(\cdot;\Lm(\ox)\big)$ is Lipschitz continuous; so we have
\begin{equation}\label{ee2}
\dist\big(\lm;\Lm(\ox)\big)-\dist\big(\lm-p_2;\Lm(\ox)\big)\le\|p_2\|=\|\Phi(x)-P\thy(\lm+\Phi(x))\|,
\end{equation}
which in combination with the obtained inequalities leads us to
\begin{equation*}
\|x-\ox\|+\dist\big(\lm;\Lm(\ox)\big)\le\ell\|\Psi(x,\lm)\|+\big(\ell(\rho+1)+1\big)\|\Phi(x)-P\thy(\lm+\Phi(x))\|.
\end{equation*}
This verifies (iv) and completes the proof of the theorem.
\end{proof}\vspace*{0.05in}

To conclude this section, let us mention some connection of the obtained characterizations of noncritical multipliers for variational systems \eqref{VS} with the uniqueness of Lagrange multipliers therein, which is {\em not} assumed in Theorem~\ref{charc}. Indeed, looking more closely at the proof of theorem reveals that the second term in \eqref{u67} is actually {\em undesired}, since it provides complications for the proof. But, as follows from Theorem~\ref{uniq}, this terms disappears (reduces to $\{0\}$) if the set of Lagrange multipliers $\Lambda(\ox)$ is a singleton. This phenomenon has been recently observed in \cite{ms18} for the case of constrained optimization problems.

\section{Noncriticality in Extended Nonlinear Programming}\label{comp-opt}

Here we concentrate on problems of composite optimization given by \eqref{co}, where $\th=\thy$ is taken from \eqref{theta}. It means that we are dealing with the class of ENLPs discussed in Section~\ref{intro}. Starting with this section we assume that $\ph_0$ and $\Phi$ are not just twice differentiable, but belongs to the class of ${\cal C}^2$-smooth mappings around the points in question.

Define the {\em Lagrangian} of \eqref{co} by
\begin{equation}\label{comlag}
L(x,\lm):=\ph_0(x)+\inp{\Phi(x)}{\lm}-\frac{1}{2}\inp{\lm}{B\lm}\;\mbox{ for }\;(x,\lm)\in\R^n\times\R^m
\end{equation}
and observe that the KKT system for \eqref{co} is written as
\begin{equation}\label{kkt-co}
\nabla_xL(x,\lm)=0,\;\lm\in\p\thy(\Phi(x)).
\end{equation}
Thus \eqref{kkt-co} is a particular case of \eqref{VS} with $\Psi:=\nabla_x L$. Denoting
\begin{equation}\label{lcom}
\Lambda_{\mathrm{com}}(\xb):=\big\{\lm\in\R^m\big|\;\nabla_xL(\xb,\lm)=0,\;\lm\in\p\thy(\Phi(\xb))\big\},
\end{equation}
the corresponding set of Lagrange multipliers, we have Definition~\ref{crit} of multiplier criticality as well as all the above results being specified for the KKT system \eqref{kkt-co}.

On the other hand, there are some phenomena concerning critical and noncritical Lagrange multipliers that distinguish KKT systems in optimization from general variational systems of type \eqref{VS}. We consider them in this and two subsequent sections.\vspace*{0.05in}

The following theorem provides a certain {\em second-order sufficient condition} ensuring simultaneously the {\em strict minimality} of a feasible solution to ENLP \eqref{co} and the {\em noncriticality} of the corresponding Lagrange multiplier. In its formulation we use the critical cone $\ok$ defined in Theorem~\ref{gdt}(iii) as well as the notation $\rge A$ for the range of a linear operator $A$. Note that the existence of Lagrange multipliers corresponding to $\ox$ in \eqref{co}, which is assumed below, is ensured by the first-order qualification condition \eqref{bcq} from Lemma~\ref{esonc}.

\begin{Th}{\bf(second-order sufficient condition for strict local minimizers and noncritical multipliers in ENLPs).}\label{esosc}
Let $(\xb,\olm)$ be a solution to KKT system \eqref{kkt-co}. Assume further that the second-order sufficient condition
\begin{equation}\label{sc}
\big\la\nabla_{xx}^2L(\xb,\olm)w,w\big\ra+2\theta_{\ok,B}\big(\nabla\Phi(\xb)w\big)>0\;\mbox{ if }\;w\in\R^n\setminus\{0\}\;\mbox{with}\;\nabla\Phi(\ox)w\in\ok^*+\rge B
\end{equation}
holds. Then there exist numbers $\ve>0$ and $\ell\ge 0$ such that the quadratic lower estimate
\begin{equation}\label{sogc}
\ph(x)\ge\ph(\ox)+\ell\,\|x-\ox\|^2\;\mbox{ for all }\;x\in\B_\ve(\ox)
\end{equation}
holds for the function $\ph$ taken from \eqref{co}. Furthermore, the Lagrange multiplier $\olm$ satisfying \eqref{sc} is noncritical for the KKT system \eqref{kkt-co}  corresponding to $\ox$.
\end{Th}
\begin{proof} Define the family of second-order difference quotients for $\ph$ at $\ox$ for $\oy\in\R^n$ by
\begin{equation}\label{lk01}
\Delta_t^2\ph(\ox,\oy)(w):=\dfrac{\ph(\ox+tw)-\ph(\ox)-t\langle\oy,\,w\rangle}{\frac{1}{2}t^2}\;\mbox{ with }\;w\in\R^{n},\;t>0.
\end{equation}
Set $\oy:=0\in\R^n$ and deduce from $\olm\in\Lcxb$ that $\oy=\nabla\ph_0(\ox)+\nabla\Phi(\ox)^*\olm$. Then for any $w\in\R^n$ we get the equalities
\begin{eqnarray*}
\Delta_t^2\ph(\ox,0)(w)&=&\Delta_t^2\ph_0(\ox,\nabla\ph_0(\ox))(w)+\dfrac{\thy\big(\Phi(\ox+tw)\big)-\thy\big(\Phi(\ox)\big)-t\langle\nabla\Phi(\ox)^*\olm,w\rangle}{\frac{1}{2}t^2}\\
&=&\Delta_t^2\ph_0(\ox,\nabla\ph_0(\ox))(w)+\dfrac{t\langle\olm,w_t\rangle-t\langle\olm,\nabla\Phi(\ox)w\rangle}{\frac{1}{2}t^2}\\
&&+\dfrac{\thy\big(\Phi(\ox)+tw_t\big)-\thy\big(\Phi(\ox)\big)-t\langle\olm,w_t\rangle}{\frac{1}{2}t^2}\\
&=&\Delta_t^2\ph_0(\ox,\nabla\ph_0(\ox))(w)+\dfrac{t\langle\olm,w_t\rangle-t\langle\olm,\nabla\Phi(\ox)w\rangle}{\frac{1}{2}t^2}+\Delta_t^2 \thy(\Phi(\ox),\olm)(w_t),
\end{eqnarray*}
where $w_t:=\nabla\Phi(\ox)w+\frac{t}{2}\langle\nabla^2\Phi(\ox)w,w\rangle+\frac{o(t^2)}{t}$. It implies together with \eqref{ssd1} and \eqref{ssd} that
\begin{eqnarray}\label{cv01}
\begin{array}{ll}
\d^2\ph(\ox,0)(w)&\ge\langle\nabla^2\ph_0(\ox)w,w\rangle+\langle\nabla^2_{xx}\langle\olm,\Phi(\ox)\rangle w,w\rangle+\d^2\thy(\Phi(\ox),\olm)\big(\nabla\Phi(\ox)w\big)\\\\
&=\langle\nabla^2_{xx}L(\ox,\olm)w,w\rangle+2\thk\big(\nabla\Phi(\ox)w\big).
\end{array}
\end{eqnarray}
Theorem~\ref{gdt}(i) tells us that $\dom\thk=(\ok\cap\ker B)^*=\ok^*+\rge B$. This means that the inclusion $\nabla\Phi(\ox)w\in\ok^*+\rge B$  amounts to $\nabla\Phi(\ox)w\in\dom\thk$. Employing the second-order sufficient condition \eqref{sc} together with \eqref{cv01} ensures that $\d^2\ph(\ox,0)(w)>0$ for all such vectors $w\in\R^n\setminus\{0\}$. Otherwise, we have $\nabla\Phi(\ox)w\notin\dom\thk$, and hence $\thk(\nabla\Phi(\ox)w)=\infty$. This along with \eqref{cv01} results in
\begin{equation*}
\d^2\ph(\ox,0)(w)>0\;\mbox{ for all }\;w\in\R^n\;\mbox{ with }\;\nabla\Phi(\ox)w\notin\dom\thk.
\end{equation*}
Combining all the above brings us to
\begin{equation*}
\d^2\ph(\ox,0)(w)>0\;\mbox{ whenever }\;w\in\R^n\setminus\{0\}.
\end{equation*}
Appealing now to \cite[Theorem~13.24]{rw} guarantees the existence of numbers $\ve>0$ and $\ell\ge 0$ for which the quadratic estimate \eqref{sogc} holds and so ensures that $\ox$ is a strict local minimizer for $\ph$.

Finally, we verify that a multiplier $\olm$ satisfying the second-order condition \eqref{sc} is noncritical for \eqref{kkt-co}. To see it, pick $(\xi,\eta)\in\R^n\times\R^m$ fulfilling \eqref{cri} with $\Psi=\nabla_x L$, i.e., so that
\begin{equation*}
\begin{cases}
\nabla^2_{xx}L(\bar{x},\bar{\lm})\xi+\nabla\Phi(\bar{x})^*\:\eta=0,\;\inp{\nabla\Phi(\bar{x})\xi-B\eta}{\eta}=0,\\
\nabla\Phi(\bar{x})\xi-B\eta\in\ok^*,\;\mbox{ and }\;\eta\in\ok.
\end{cases}
\end{equation*}
It follows from $\nabla\Phi(\bar{x})\xi-B\eta\in\ok^*$ and the discussion above that $\nabla\Phi(\bar{x})\xi\in\dom\thk$ and that
$\eta\in\partial\thk(\nabla\Phi(\ox)\xi)$. Employing the subdifferential expression in \eqref{fo} gives us
\begin{equation*}
\thk\big(\nabla\Phi(\ox)\xi\big)=\langle\eta,\nabla\Phi(\ox)\xi\rangle-\frac{1}{2}\langle B\eta,\eta\rangle=\frac{1}{2}\langle B\eta,\eta\rangle.
\end{equation*}
In this way we arrive at the equalities
\begin{eqnarray*}
0=\langle\nabla^2_{xx}L(\bar{x},\bar{\lm})\xi,\xi\rangle+\langle\eta,\nabla\Phi(\bar{x})\xi\rangle&=&\langle\nabla^2_{xx}L(\bar{x},\bar{\lm})\xi,
\xi\rangle+\langle B\eta,\eta\rangle\\
&=&\langle\nabla^2_{xx}L(\bar{x},\bar{\lm})\xi ,\xi\rangle+2\thk(\nabla\Phi(\ox)\xi),
\end{eqnarray*}
which yield $\xi=0$ due to \eqref{sc} as well as to $\nabla\Phi(\bar{x})\xi\in\dom\thk=\ok^*+\rge B$. This shows that $\olm$ is a noncritical multiplier of \eqref{kkt-co} corresponding to $\ox$ and thus completes the proof.
\end{proof}\vspace*{0.05in}

The next example, which revisits Example~\ref{ex1} in the ENLP framework, illustrates the possibility to use the second-order sufficient condition \eqref{sc} to justify the strict optimality of a feasible solution to \eqref{co} and the noncriticality of the corresponding Lagrange multiplier.

\begin{Example}{\bf(multiplier noncriticality via the second-order sufficient condition).}\label{ex1a} {\rm Consider the ENLP from \eqref{co}, where $m=n$, $\ph_0(x):=x_1^2+\ldots+x_n^2$ and $\Phi(x):=x$, and where $Y$ and $B$ are taken from Example~\ref{ex1}. Then we have
\begin{equation}\label{theta-phi}
\begin{array}{c c c}
\thy\big(\Phi(x)\big)&=&\disp\sup\limits_{y\in\Rn_+}\Big\{\inp{y}{\Phi(x)}-\frac{1}{2}\inp{y}{y}\Big\}\\
&=&\disp\sup\limits_{(y_1,\ldots,y_n)\in\Rn_+}\Big\{\sum\limits_{i=1}^{n}\big(x_iy_i-\frac{1}{2}y_i^2\big)\Big\}\\
&=&\disp\frac{1}{2}\sum\limits_{i=1}^n\big(\max\{x_i,0\}\big)^2.
\end{array}
\end{equation}

Let us check that condition \eqref{sc} holds when $\ox=0$ and $\olm=0$, which confirms by Theorem~\ref{esosc} that $\ox$ is a strict minimizer for this ENLP and $\olm$ is the corresponding noncritical multiplier. Indeed, it follows from Example~\ref{ex1} that $\olm\in\partial\theta(\zb)$, where $\zb:=\Phi(\xb)=0$. By the structure of $L(x,\lm)$ we have the expressions
\begin{equation*}
\nabla_xL(x,\lm)=(2x_1+\lm_1,\ldots,2x_n+\lm_n)\;\mbox{ and }\;\nabla^2_{xx}L(x,\lm)=2I.
\end{equation*}
Then $\nabla_xL(\xb,\olm)=0$ and hence $\olm\in\Lcxb$. Since $\rge B=\Rn$, it follows that $\{w|\;\nabla\Phi(\xb)w\in\mcK^*+\rge B\}=\Rn$, and therefore the sufficient condition in Theorem~\ref{esosc} reads as
\begin{equation*}
\inp{\nabla_{xx}^2L(\xb,\olm)\xi}{\xi}+2\theta_{\ok,B}\big(\nabla\Phi(\xb)\xi\big)>0\;\textrm{ for all }\;\xi\ne 0,
\end{equation*}
which is equivalently presented by
\begin{equation}
2\inp{\xi}{\xi}+2\theta_{\ok,B}\big(\nabla\Phi(\xb)\xi\big)> 0\;\textrm{ for all }\;\xi\ne 0.
\end{equation}
Furthermore, Example~\ref{ex1} tells us that $\mcK=\Rn_+\cap\{\zb\}^\perp$ and so $\mcK=\Rn_+=Y$. Combining this with \eqref{theta-phi}, the sufficient condition \eqref{sc} now becomes
\begin{equation}\label{sc1}
2\inp{\xi}{\xi}+2\thy\big(\nabla\Phi(\xb)\xi\big)>0\;\textrm{ for all }\;\xi\ne 0.
\end{equation}
Since $\thy$ from \eqref{theta-phi} is always nonnegative, condition \eqref{sc1} holds, and thus it confirms the strict minimality of $\xb$ and the noncriticality of $\olm$.}
\end{Example}

\section{Critical Multipliers and Full Stability of Minimizers in ENLPs}\label{full-stab}

This section also deals with constrained minimization problems of the ENLP type and delivers as important message for both theoretical and numerical aspects of optimization. As discussed in Section~\ref{intro}, critical multipliers are particularly responsible for slow convergence of major primal-dual algorithms of optimization and are desired to be excluded for a given local minimizer. It is natural to suppose that seeking not arbitrary while just ``nice" and stable in some sense local minimizers allows us to rule out the appearance of critical multipliers associated with such local optimal solutions. It is conjectured in \cite{m15} that fully stable local minimizers in the sense of \cite{lpr} are appropriate candidate for excluding critical multipliers. This conjecture is affirmatively verified in \cite{brs13} for problems \eqref{co} with $\theta=\thy$ where $B=0$. Now we are able to extend this result to the general case of \eqref{theta} with an arbitrary symmetric positive-semidefinite matrix $B$.

To proceed, we first specify the definition of fully stable local minimizers from \cite{lpr} for problems \eqref{co} with term \eqref{theta}. Consider their {\em canonically perturbed} version described by
\begin{equation}\label{pertco}
\textrm{minimize }\;\ph_0(x)+\theta\big(\Phi(x)+p_2\big)-\inp{p_1}{x}\;\textrm{ subject }\;x\in\R^n
\end{equation}
with parameter pairs $(p_1,p_2)\in\R^n\times\R^m$. Fix $\gamma>0$ and $(\xb,\pb_1,\pb_2)$ with $\Phi(\xb)+\pb_2\in\dom\theta$ and then define the parameter-depended optimal value function for \eqref{pertco} by
\begin{equation*}
m_\gamma (p_1,p_2):=\inf_{\n{x-\xb}\le\gamma}\big\{\ph_0(x)+\theta\big(\Phi(x)+p_2\big)-\inp{p_1}{x}\big\}
\end{equation*}
together with the parameterized set of optimal solutions to \eqref{pertco} given by
\begin{equation}\label{arg}
M_\gamma (p_1,p_2):=\argmin_{\n{x-\xb}\le\gamma}\big\{\ph_0(x)+\theta\big(\Phi(x)+p_2)-\inp{p_1}{x}\big\}
\end{equation}
with the convention that $\argmin:=\emp$ when the expression under minimization in \eqref{arg} is $\infty$. We say that $\xb$ is a {\em fully stable} local optimal solution to problem \eqref{co} if there exist a number $\gamma>0$ and neighborhoods $U$ of $\pb_1$ and $W$ of $\pb_2$ such that the mapping $(p_1,p_2)\mapsto M_\gamma(p_1,p_2)$ is single-valued and Lipschitz continuous with $M_\gamma(\pb_1,\pb_2)=\{\xb\}$ and that the function $(p_1,p_2)\mapsto m_\gamma(p_1,p_2)$ is likewise Lipschitz continuous on $U\times W$.

Note that \cite[Proposition~3.5]{lpr} deduces the local Lipschitz continuity of $m_\gamma$ from the {\em basic constraint qualification} \eqref{bcq} formulated in the following lemma, which is obtained in \cite[Exercise~13.26]{rw}. The second-order necessary condition presented below can be viewed as a ``no-gap" version of the second-order sufficient one used in Theorem~\ref{esosc} with the notation therein.

\begin{Lemma}{\bf(second-order necessary optimality condition for composite optimization problems).}\label{esonc}
Let $\xb$ be a local optimal solution to problem \eqref{co} with $\theta=\thy$ taken from \eqref{theta}, and let the basic constraint qualification
\begin{equation}\label{bcq}
N_{\ss{\dom\thy}}(\Phi(\ox))\cap\ker\nabla\Phi(\xb)^*=\{0\}
\end{equation}
be satisfied, and so $\Lambda_{\mathrm{com}}(\xb)\ne\emp$. Then we have second-order necessary optimality condition
\begin{equation}\label{nc1}
\max_{\lm\in\Lambda_{\mathrm{com}}(\xb)}\big\la\nabla_{xx}^2L(\xb,\lm)w,w\big\ra+2\theta_{\ok,B}\big(\nabla\Phi(\xb)w\big)\ge 0
\end{equation}
valid for all $w\in\R^n$ with $\nabla\Phi(\ox)w\in\ok^*+\rge B$.
\end{Lemma}

Now we are ready to establish the aforementioned result in the general ENLP setting.

\begin{Th}{\bf(excluding critical multipliers by full stability of local minimizers).}\label{criful} Let $\xb$ be a fully stable local optimal solution to problem \eqref{co}, and let $\theta$ be taken from \eqref{theta}. Then the Lagrange multiplier set $\Lcxb$ in \eqref{lcom} is nonempty and does not include critical multipliers.
\end{Th}
\begin{proof}
First we show that the full stability of $\xb$ ensures the validity of the qualification condition \eqref{bcq}. Indeed, pick any $\eta\in N_{\ss{\dom\thy}}(\Phi(\ox))\cap\ker\nabla\Phi(\xb)^*$. Select $p_1=\pb_1:=0$ and $p_2:=t\eta$ as $t\downarrow 0$. It follows from the full stability of $\xb$ that there exist a Lipschitz constant $\ell\ge 0$ and the unique solution $x_{p_1 p_2}$ to problem \eqref{pertco} such that
\begin{equation}\label{xleta}
\n{x_{p_1 p_2}-\xb}\le\ell t\n{\eta}.
\end{equation}
Since $\Phi(x_{p_1 p_2})+p_2\in\dom\thy$ and $\eta\in N_{\ss{\dom\thy}}(\Phi(\ox))$, we get $\inp{\eta}{\Phi(x_{p_1 p_2})+p_2-\Phi(\xb)}\le 0$.
This  gives us the relationships
\begin{equation*}
\begin{array}{ll}
0&\ge\big\la\eta,\nabla\Phi(\xb)(x_{p_1 p_2}-\xb)+o(\n{x_{p_1 p_2}-\xb})+p_2\big\ra\\
&=\big\la\nabla\Phi(\xb)^*\eta,x_{p_1 p_2}-\xb\big\ra+\big\la\eta,o(\n{x_{p_1 p_2}-\xb})+p_2\big\ra\\
&=\big\la\eta,o(\n{x_{p_1 p_2}-\xb})\big\ra+t\|\eta^2\|.
\end{array}
\end{equation*}
Using estimate \eqref{xleta} and letting $t\dn 0$ lead to $\eta=0$. Thus the basic constraint qualification \eqref{bcq} is satisfied, which ensures that $\Lambda_{\mathrm{com}}(\xb)\ne\emp$.

Next we pick any $\olm\in\Lcxb$ and show that it is noncritical for the unperturbed KKT system \eqref{kkt-co} corresponding to $\ox$. Consider the KKT system for the perturbed problem \eqref{pertco} that can be written as
\begin{equation}\label{kkt}
\begin{pmatrix}
p_1\\p_2
\end{pmatrix}\in\begin{pmatrix}
\nabla_x L(x,\lm)\\
-\Phi(x)\end{pmatrix}+\begin{pmatrix}
0\\
(\partial\thy)^{-1}(\lm)
\end{pmatrix}.
\end{equation}
Let $S_{KKT}\colon\R^n\times\R^m\rightrightarrows\R^n\times\R^m$ be the solution map to \eqref{kkt} given by
\begin{equation}\label{skkt}
S_{KKT}(p_1,p_2):=\big\{(x,\lm)\in\R^n\times\R^m\big|\;p_1=\nabla_x L(x,\lm),\;\lm\in\p\thy\big(p_2+\Phi(x)\big)\big\}.
\end{equation}
Employing Theorem~\ref{charc}, we only need to prove that there exist numbers $\epsilon>0$ and $\ell\ge 0$ as well as neighborhoods $U$ of $0\in\mathbb{R}^n$ and $W$ of $0\in\mathbb{R}^m$ such that for any $(p_1,p_2)\in U\times V$ and any $(x_{p_1 p_2},\lm_{p_1 p_2})\in S_{KKT}(p_1,p_2)\cap(\mathbb{B}_\epsilon(\xb)\times\mathbb{B}_\epsilon(\olm))$, estimate \eqref{upper} holds with replacing $\Lambda(\xb)$ by the set of Lagrange multipliers $\Lcxb$ taken from \eqref{lcom}.

To this end we deduce from the full stability of $\ox$ in \eqref{pertco} with $(\pb_1,\pb_2)=(0,0)$ due to the result of \cite[Proposition~6.1]{brs13} that there exist neighborhoods $\tilde{U}\times\tilde{W}$ of $(0,0)$ and $\tilde{V}$ of $\xb$ for which the set-valued mapping
\begin{equation*}
(p_1,p_2)\mapsto Q(p_1,p_2):=\big\{x\in\mathbb{R}^n\big|\;p_1\in\nabla\ph_0(x)+\nabla\Phi(x)^*\p\thy(\Phi(x)+p_2)\big\}
\end{equation*}
admits a Lipschitzian single-valued graphical localization on $\tilde{U}\times\tilde{W}\times\tilde{V}$. This means that there exists a Lipschitzian single-valued mapping $g\colon\tilde{U}\times\tilde{W}\mapsto\tilde{V}$ such that $(\gph Q)\cap(\tilde{U}\times\tilde{W}\times\tilde{V})=\gph g$. Denote $U:=\tilde{U}$, $W:=\tilde{W}$ and take $\epsilon>0$ so small that $\mathbb{B}_\epsilon(\xb)\subset\tilde{V}$. The Lipschitzian single-valued graphical localization property of $Q$ allows us to find a constant $\ell\ge 0$ such that for any $(p_1,p_2)\in U\times W$ and any $(x_{p_1 p_2},\lm_{p_1 p_2})\in S_{KKT}(p_1,p_2)\cap\big(\mathbb{B}_\epsilon(\xb)\times\mathbb{B}_\epsilon(\olm)\big)$ we have the inclusion $x_{p_1 p_2}\in Q(p_1,p_2)$, and hence
\begin{equation*}
\n{x_{p_1 p_2}-\xb}=\n{x_{p_1 p_2}-x_{\pb_1\pb_2}}\le\ell\big(\n{p_1}+\n{p_2}\big).
\end{equation*}
Using now the error bound estimate \eqref{pap} from the proof of Theorem~\ref{charc} with replacing $\Lambda(\xb)$ by $\Lcxb$ and adjusting $\epsilon$ if necessary give us the semi-isolated calmness property \eqref{upper}, which is equivalent to the noncriticality of $\olm$ that was chosen arbitrary from the Lagrange multiplier set $\Lcxb$. This therefore completes the proof of theorem.
\end{proof}\vspace*{0.05in}

The result of Theorem~\ref{criful} calls for the deriving verifiable conditions for full stability of local minimizers to \eqref{co} expressed entirely via the problem data and the given minimizer. Such conditions allow us to efficiently exclude slow convergence of primal-dual algorithms to seek fully stable minimizers based on the initial data. Some characterizations of full stability of local minimizers for ENLPs of type \eqref{co} are obtained in \cite[Theorem~7.3]{brs13} under rather strong assumptions. Relaxing these assumptions is a challenging goal of our future research.

\section{Noncriticality and Lipschitzian Stability of Solutions to ENLPs}\label{lip-stab}

In this section we use the machinery developed above to investigate other notions of Lipschitzian stability, which occur to be related to noncriticality of multipliers for ENLPs. The following theorem provides characterizations of both isolated calmness and robust isolated calmness properties of the KKT solution map \eqref{skkt} associated with ENLP \eqref{co} in terms of the second-order sufficient condition \eqref{sc} as well as noncriticality and uniqueness of Lagrange multipliers.

\begin{Th}{\bf(characterizations of robust isolated calmness of solution maps).}\label{calmness} Let $\xb$ be a feasible solution to ENLP \eqref{co} with $\theta$ taken from \eqref{theta}, and let $\olm\in\Lcxb$ be a corresponding Lagrange multiplier from \eqref{lcom}. The following assertions are equivalent:\vspace{-0.2 cm}
\begin{itemize}[noitemsep]
\item[{\bf(i)}] The solution map $S_{KKT}$ from \eqref{skkt} is robustly isolatedly calm at the point $\big((0,0),(\xb,\olm)\big)\in\R^{n+m}\times\R^{n+m}$,
and $\xb$ is a local optimal solution to \eqref{co}.

\item[{\bf(ii)}] The second-order sufficient condition \eqref{sc} holds, and $\Lambda_{\mathrm{com}}(\xb)=\{\olm\}$.

\item[{\bf(iii)}] $\Lambda_{\mathrm{com}}(\xb)=\{\olm\}$, $\xb$ is a local optimal solution to \eqref{co}, and $\olm$ is a noncritical multiplier for
\eqref{VS} with $\Psi=\nabla_x L$ that is associated with the optimal solution $\xb$.

\item[{\bf(iv)}] $S_{KKT}$ is isolatedly calm at $\big((0,0),(\xb,\olm)\big)$, and $\xb$ is a local optimal solution to \eqref{co}.
\end{itemize}
\end{Th}
\begin{proof} The outline of the proof is as follows. We sequentially verify implications (ii)$\Longrightarrow$(iii), (iii)$\Longrightarrow$(iv),
(iv)$\Longrightarrow$(iii), (iii)$\Longrightarrow$(ii), and (i)$\iff$(iv).

To prove (ii)$\Longrightarrow$(iii), assume the validity of \eqref{sc} and that $\Lcxb=\{\olm\}$. Then Theorem~\ref{esosc} tells us that $\xb$ is a strict local minimizer of \eqref{co} and that $\olm$ is a noncritical multiplier of \eqref{VS} with $\Psi=\nabla_x L$ corresponding to $\ox$, and thus (iii) is satisfied.

Suppose next that all the conditions in (iii) hold.
Since $\olm$ is noncritical, we derive the semi-isolated calmness of $S_{KKT}$ at $\big((0,0),(\xb,\olm)\big)$. This together with $\Lcxb=\{\olm\}$
results in the existence of   a number $\ell\ge 0$ as well as neighborhoods $U$ of $(0,0)$ and $V$ of $(\xb,\olm)$ such that
\begin{equation}\label{ilc}
S_{KKT}(p_1,p_2)\cap V\subset\big\{(\xb,\olm)\big\}+\ell\n{(p_1,p_2)}\mathbb{B}\;\textrm{ for all }\;(p_1,p_2)\in U.
\end{equation}
Thus $S_{KKT}$ enjoys the isolated calmness property at $\big((0,0,(\xb,\olm))\big)$, and we arrive at (iv).

To verify the opposite implication (iv)$\Longrightarrow$(iii), let us show that the isolated calmness of $S_{KKT}$ at $\big((0,0),(\xb,\olm)\big)$ in (iv) yields $\Lcxb=\{\olm\}$.  Indeed, suppose on the contrary that $\Lcxb$ is not a singleton. Then there exists $\hat{\lm}\in\Lcxb$ with $\hat{\lm}\ne\olm$. Since the set $\Lcxb$ is convex, every point of the line segment connecting $\olm$ and $\hat{\lm}$ belongs to $\Lcxb$. The isolated calmness of $S_{KKT}$ at $\big((0,0),(\xb,\olm)\big)$ amounts to \eqref{ilc}, and hence we can find $\lm'\ne\olm$ with $\lm'\in\Lcxb$ and such that $\lm'$ is sufficiently close to $\olm$, i.e., $(\xb,\lm')\in V$. Then it follows from \eqref{ilc} that
\begin{equation*}
\n{\lm'-\olm}\le\ell\cdot 0=0,
\end{equation*}
which yields $\lm'=\olm$, a contradiction ensuring that $\Lcxb$ is a singleton. Theorem~\ref{charc} tells us that $\olm$ is a noncritical multiplier of \eqref{VS} corresponding to $\ox$, and thus (iii) holds.

Next we verify implication (iii)$\Longrightarrow$(ii). Let us first deduce from $\Lcxb=\{\olm\}$ in (iii) that the qualification condition \eqref{bcq} in (ii) is satisfied. Supposing the contrary, find a normal $v\in N_{{\rm dom}\,\thy}(\Phi(\ox))$ with $v\ne 0$ such that $\nabla\Phi(\xb)^*v=0$. Letting $\lm':=\olm+v$, we get $\lm'\ne\olm$ and $\nabla_x L(\xb,\lm')=0$ for the Lagrangian function \eqref{comlag}. By the choice of $v$ and the normal cone definition \eqref{nc-conv} we get from the above that
\begin{equation*}
\la\lm',z-\Phi(\ox)\ra\le\thy(z)-\thy(\Phi(\ox))\;\mbox{ for all }\;z\in\dom\thy,
\end{equation*}
which shows that $\lm'\in\p\thy(\Phi(\ox))$ and hence $\lm'\in\Lcxb$ due to $\nabla_x L(\xb,\lm')=0$. Since $\lm'\ne\vb$, it gives us a contradiction with the assumption of $\Lcxb=\{\olm\}$ in (iii) and thus justifies the validity of the qualification condition \eqref{bcq}. Employing now Lemma~\ref{esonc} tells us that the second-order {\em necessary} optimality condition \eqref{nc1} is satisfied.

To finish the verification of (iii)$\Longrightarrow$(ii), we need to prove that the second-order {\em sufficient} optimality condition \eqref{sc} holds under the assumptions in (iii). Supposing the contrary gives us a nonzero element $\xi_0\in\{w|\;\nabla\Phi(\xb)w\in\ok^*+\rge B\}$ such that
\begin{equation*}
\big\la\nabla_{xx}^2L(\xb,\olm)\xi_0,\xi_0\big\ra+2\theta_{\ok,B}\big(\nabla\Phi(\xb)\xi_0\big)\le 0.
\end{equation*}
Since $\Lcxb=\{\olm\}$, it is easy to see that the second-order necessary condition \eqref{nc1} can be equivalently written as
\begin{equation*}
\big\la\nabla_{xx}^2L(\xb,\olm)w,w\big\ra+2\theta_{\ok,B}\big(\nabla\Phi(\xb)w\big)\ge 0\quad\mbox{for all}\;w\in\R^n\;\;\mbox{with}\;\;\nabla\Phi(\ox)w\in \dom\thk.
\end{equation*}
Furthermore, employing the equalities
\begin{equation*}
\nabla\Phi(\xb)\xi_0\in\ok^*+\rge B=\big(\ok\cap\ker B\big)^*=\dom\thk
\end{equation*}
allows us to deduce from the equivalent form of the second-order necessary condition that
\begin{equation*}
\big\la\nabla_{xx}^2L(\xb,\olm)\xi_0,\xi_0\big\ra+2\theta_{\ok,B}\big(\nabla\Phi(\xb)\xi_0\big)=0.
\end{equation*}
This in turn implies that the vector $\xi_0$ is an {\em optimal solution} to the problem
\begin{equation*}
\min_{\xi\in\R^n}\;\frac{1}{2}\big\la\nabla_{xx}^2L(\xb,\olm)\xi,\xi\big\ra+\theta_{\ok,B}\big(\nabla\Phi(\xb)\xi\big).
\end{equation*}
Applying the subdifferential Fermat rule to the latter problem and then using the elementary sum rule for convex subgradients together with the chain rule from
\cite[Exercise 10.22(b)]{rw} yield
\begin{eqnarray*}
0&\in&\nabla_{xx}^2L(\xb,\olm)\xi_0+
\nabla\Phi(\xb)^*\p\thk(\nabla\Phi(\xb)\xi_0)\\
&=& \nabla_{xx}^2L(\xb,\olm)\xi_0+
\nabla\Phi(\xb)^*D\p\thy(\Phi(\ox),\olm)(\nabla\Phi(\xb)\xi_0),
\end{eqnarray*}
where the last equality comes from  \eqref{gdr}. Since $\xi_0\ne 0$, it shows by Definition~\ref{crit} that $\olm$ is a critical multiplier.
This contradicts the assumption in (iii) that $\olm$ is a noncritical multiplier and therefore verifies the validity of \eqref{sc} and the entire implication (iii)$\Longrightarrow$(ii).
	
Our next step is to prove implication (i)$\Longrightarrow$(iv), which clearly holds. To complete the proof of the theorem, it remains to verify implication (iv)$\Longrightarrow$(i).
To achieve this implication, we only need to show that there are neighborhoods $U$ of $(0,0)$ and $V$ of $(\xb,\olm)$
such that $S_{KKT}(p_1,p_2)\cap V\neq \emptyset$ for all $(p_1,p_2)\in U$.
To this end, define the set-valued mapping $Q\colon\mathbb{R}^m\rightrightarrows\Rn$ by
\begin{equation*}
Q(p):=\big\{x\in\R^n\big|\;\Phi(x)+p\in\dom\thy\big\},\quad p\in\Rn.
\end{equation*}
Having already proved (iv) and (iii) are equivalent, we have
the  qualification condition \eqref{bcq} because of the assumptions in (iii).
As proved above, (iii) and  (ii) are equivalent. Thus the second-order sufficient condition \eqref{sc} is satisfied and implies by Theorem~\ref{esosc} that $\xb$ is a strict local minimizer for \eqref{co}.
This gives a neighborhood $O$ of $\ox$ for which we have
\begin{equation}\label{gd03}
\ph_0(\ox)+\thy\big(\Phi(\ox)\big) < \ph_0(x)+\thy\big(\Phi(x)\big)\quad \mbox{for all}\quad x\in O.
\end{equation}
Applying \cite[Theorem~4.37(ii)]{m06} to the mapping $Q$ with the initial point $(0,\ox)$ gives us numbers $r>0$ and $\ell\ge 0$ such that
\begin{equation}\label{q}
Q(p)\cap\B_r(\xb)\subset Q(p')+\ell\n{p-p'}\B\;\textrm{ for all }\;p,p'\in \B_r(0),
\end{equation}
where $r$ can be chosen such that $\B_r(\xb)\subset O$. Consider now the optimization problem
\begin{equation}\label{mf01}
\textrm{minimize }\;\ph_0(x)+\thy\big(\Phi(x)+p_2\big)-\inp{p_1}{x}\;\textrm{ subject to }\;x\in\B_r(\xb)\cap Q(p_2).
\end{equation}
It is clear that this problem admits an optimal solution $x_{p_1p_2}$ for any pair $(p_1,p_2)\in \R^n\times \B_r(0)$ since the cost function therein is lower semicontinuous while the constraint set is obviously compact.
Let us now show that there is a number $\epsilon>0$ with $\B_\epsilon(0,0)$ such that
\begin{equation}\label{bd}
x_{p_1p_2}\in{\rm int}\,\B_r(\xb)\;\textrm{ for any }\;(p_1,p_2)\in \B_\epsilon(0,0).
\end{equation}
Suppose the contrary and then find sequences $(p_{1k},p_{2k})\rightarrow(0,0)$ and $x_{p_{1k}p_{2k}}$ for which $\|{x_{p_{1k}p_{2k}}}-\ox\|=r$.
We get without loss of generality that $x_{p_{1k}p_{2k}}\rightarrow x_0$ as $k\to\infty$ and so $\|{x_0}-\ox\|=r$.
This yields $x_0\neq \ox$.
Since $x_{p_{1k}p_{2k}}$ is an optimal solution to \eqref{mf01}, it follows that
\begin{equation}\label{coarr}
\ph_0(x_{p_{1k}p_{2k}})+\thy\big(\Phi(x_{p_{1k}p_{2k}})+p_{2k}\big)-\inp{p_{1k}}{x_{p_{1k}p_{2k}}}\le\ph_0(x)+\thy\big(\Phi(x)+p_{2k}\big)-
\inp{p_{1k}}{x}
\end{equation}
for all $x\in\B_r(\xb)\cap Q(p_{2k})$. Pick any $x\in\B_{\frac{r}{2}}(\xb)\cap Q(0)$ and $k\in\mathbb{N}$ so large that $p_{2k}\in\alpha\B$ with $\alpha<\min\{\frac{r}{2\ell},r\}$. It follows from \eqref{q} that there exist $x'\in Q(p_{2k})$ and $b\in\B$ satisfying
\begin{equation*}
\n{x'-\xb}\le\n{x-\xb}+\ell\n{p_{2k}}\le\frac{r}{2}+\ell\frac{r}{2\ell}=r,\;\textrm{ where }\;x:=x'+\ell\n{p_{2k}}b.
\end{equation*}
Thus $x'\in\B_r(\xb)\cap Q(p_{2k})$, and it follows from \eqref{coarr} that
\begin{eqnarray*}
\ph_0(x_{p_{1k}p_{2k}})&+\thy\big(\Phi(x_{p_{1k}p_{2k}})+p_{2k}\big)-\inp{p_{1k}}{x_{p_{1k}p_{2k}}}\le\ph_0\big(x-\ell\n{p_{2k}}b\big)\\
&+\thy\big(\Phi(x-\ell\n{p_{2k}}b)+p_{2k}\big)-\inp{p_{1k}}{x-\ell\n{p_{2k}}b}.
\end{eqnarray*}
Passing to the limit at the latter inequality as $k\rightarrow\infty$ gives us the estimate
\begin{equation*}
\ph_0(x_0)+\thy\big(\Phi(x_0)\big)\le\ph_0(x)+\thy\big(\Phi(x)\big),
\end{equation*}
which holds for all $x\in\B_{\frac{r}{2}}(\xb)\cap Q(0)$. In particular,  we have
\begin{equation}\label{oppstos}
\ph_0(x_0)+\thy\big(\Phi(x_0)\big)\le\ph_0(\xb)+\thy\big(\Phi(\xb)\big),
\end{equation}
which contradicts \eqref{gd03} since $x_0\neq \ox$ and $x_0\in \B_r(\ox)\subset O$, and thus we arrive at \eqref{bd}.

At the last step of the proof, denote by ${\Lambda}_\mathrm{com}(x_{p_1p_2})$ be the set of Lagrange multipliers associated with the optimal solution $x_{p_1p_2}$ to problem \eqref{mf01}. It follows from the validity of the qualification condition \eqref{bcq} and its robustness with respect to perturbations of the initial point that this qualification condition is also satisfied for the perturbed problem \eqref{mf01}.
This implies in turn that ${\Lambda}_\mathrm{com}(x_{p_1p_2})\ne\emp$ for all $(p_1,p_2)$  sufficiently close to $(0,0)\in \R^n\times \R^m$.
Assume without loss of generality that ${\Lambda}_\mathrm{com}(x_{p_1p_2})\ne\emp$ for all $(p_1,p_2)\in \B_\epsilon(0,0)$, where $\epsilon$ is taken from
\eqref{bd}.
Using  a similar argument as \eqref{err} and \eqref{pap} via the Hoffman lemma
gives us a constant $\ell'\ge 0$ such that for any $(p_1,p_2)\in \B_\epsilon(0,0)$ and any $\lm_{p_1p_2}\in {\Lambda}_{\mathrm{com}}(x_{p_1p_2})$ we have
\begin{equation*}
\|\lm_{p_1p_2}-\olm\|=\dist\big(\lm_{p_1p_2};\Lambda_\mathrm{com}(\bar{x})\big)\le\ell'\big(\n{x_{p_1p_2}-\bar{x}}+\n{p_1}+\n{p_2}\big).
\end{equation*}
This clearly proves the existence of a neighborhood $V$ of $(\xb,\olm)$
such that $S_{KKT}(p_1,p_2)\cap V\neq \emptyset$ for all $(p_1,p_2)\in  \B_\epsilon(0,0)$
and so finishes the proof of implication (iv)$\Longrightarrow$(i).
\end{proof}\vspace*{0.05in}

The final piece of this paper concerns yet another well-recognized Lipschitzian type property, which seems to be the most natural extension of {\em robust} Lipschitzian behavior to set-valued mapping. For this reason we label it as the Lipschitz-like property \cite{m06} while it is also known as the pseudo-Lipschitz or Aubin one. It is said that a set-valued mapping/multifunction $F\colon\R^n\tto\R^m$ is {\em Lipschitz-like} around $(\ox,\oy)\in\gph F$ if there exists a constant $\ell\ge 0$ together with neighborhoods $U$ of $\ox$ and $V$ of $\oy$ such that we have the inclusion
\begin{equation}\label{lip-like}
F(x')\cap V\subset F(x)+\ell\|x-x'\|\B\;\mbox{ for all }\;x,x'\in U.
\end{equation}
To formulate a convenient characterization of property \eqref{lip-like}, we recall first the notion of the {\em normal cone} to a set $\O\subset\R^n$ at a point $\ox\in\O$ defined by
\begin{equation*}
N_\Omega(\ox):=\Big\{v\in\R^n\Big|\;\textrm{ there exist }\;x_k\xrightarrow{\Omega}\ox,\;v_k\to v\;\mbox{ with }\;\disp\limsup_{x\to x_k}\frac{\la v_k,x-x_k\ra}{\|x_k-x\|}\le 0\Big\}.
\end{equation*}
The {\em coderivative} of a set-valued mapping $F\colon\R^n\tto\R^m$ at $(\ox,\oy)\in\gph F$ is given by
\begin{equation*}
D^*F(\xb,\yb)(v):=\big\{u\in\R^n\big|\;(u,-v)\in N_{\gph F}(\xb,\yb)\big\},\quad v\in\R^m.
\end{equation*}
The following characterization of the Lipschitz-like property for any closed-graph mapping $F\colon\R^n\tto\R^m$ around $(\ox,\oy)\in\gph F$ is known as the {\em Mordukhovich criterion} from \cite[Theorem~9.40]{rw}, where the proof is different from the original one; see \cite[Theorem~5.7]{m93} as well as its infinite-dimensional extension given in \cite[Theorem~4.10]{m06}:
\begin{equation}\label{cod-cr}
D^*F(\ox,\oy)(0)=\{0\}.
\end{equation}
Note the results obtained therein provide also a precise computation of the {\em exact bound}/infimum of Lipschitzian moduli $\{\ell\}$ in \eqref{lip-like} via the coderivative norm at $(\ox,\oy)$.

Full {\em coderivative calculus} developed for coderivatives, which is based on variational/extremal principles of variational analysis and can be found in \cite{m18,m06,rw}, allows us apply the general characterization \eqref{cod-cr} to specific multifunctions given in some structural forms. The next theorem employs \eqref{cod-cr} and coderivative calculus to characterize the Lipschitz-like property of the solution map \eqref{skkt} to the canonically perturbed KKT system \eqref{kkt}.

\begin{Th}{\bf(Lipschitz-like property of solution maps).}\label{Liplike} Let $(\xb,\olm)\in S_{KKT}(0,0)$ for the solution map $S_{KKT}$ defined in \eqref{skkt} with $\theta$ taken from \eqref{theta}. Then $S_{KKT}$ is Lipschitz-like around $\big((0,0),(\xb,\olm)\big)$ if and only if we have the implication
\begin{equation}\label{Lipdesc}
\begin{cases}
\nabla_{xx}^2 L(\xb,\olm)\xi+\nabla\Phi(\xb)^*\eta=0\\
\eta\in\big(D^*\partial\thy\big)(\Phi(\ox),\olm)\big(\nabla\Phi(\xb)\xi\big)
\end{cases}
\Longrightarrow(\xi,\eta)=(0,0).
\end{equation}
\end{Th}
\begin{proof} Consider the mapping $G$ from  \eqref{g} with $\Psi=\nabla_x L$.
We easily deduce from the coderivative definition and the form of $S$ that
\begin{equation}\label{S}
(\xi,\eta)\in D^*S_{KKT}\big((0,0),(\xb,\olm)\big)(w_1,w_2)\Longleftrightarrow-(w_1,w_2)\in D^*G\big((\xb,\olm),(0,0)\big)(-\xi,-\eta)
\end{equation}
for all $(\xi,\eta)\in\mathbb{R}^n\times\mathbb{R}^m$ and $(w_1,w_2)\in\mathbb{R}^n\times\mathbb{R}^m$.
Using the structure of $G$ and employing the coderivative sum rule in the equality form from \cite[Theorem~3.9]{m18} yield
\begin{equation}\label{G}
\begin{array}{ll}
D^*G\big((\xb,\olm),(0,0)\big)(\xi,\eta)&=\left[\begin{array}{c c}\nabla_{xx}^2 L(\xb,\olm)&-\nabla\Phi(\xb)^*\\
\nabla\Phi(x)&0
\end{array}\right]
\left[\begin{array}{c c}\xi\\\eta\end{array}\right]+\left[\begin{array}{c c}0\\D^*(\partial\thy)^{-1}(\olm,\Phi(\ox))(\eta) \end{array}\right]\\\\
&=\left[\begin{array}{c c}\nabla_{xx}^2 L(\xb,\olm)\xi-\nabla\Phi(\xb)^*\eta\\\nabla\Phi(x)\xi+D^*(\partial\thy)^{-1}(\olm,\Phi(\ox))(\eta)
\end{array}\right].
\end{array}
\end{equation}
It follows from \eqref{S} and the coderivative criterion \eqref{lip-like} that $S_{KKT}$ is Lipschitz-like around $\big((0,0),(\xb,\olm)\big)$ if and only if we have the implication
\begin{equation*}
(0,0)\in D^*G\big((\xb,\olm),(0,0)\big)(\xi,\eta)\Longrightarrow(\xi,\eta)=(0,0),
\end{equation*}
which leads us together the coderivative representation for $G$ in \eqref{G} to characterization \eqref{Lipdesc} of the Lipschitz-like property of the solution map $S_{KKT}$.
\end{proof}\vspace*{0.05in}

Combining finally the obtained characterization of the Lipschitz-like property in Theorem~\ref{Liplike} with some known facts of variational analysis allows us to reveal a relationship between the latter property of the solution map $S_{KKT}$ and its isolated calmness at the same point.

\begin{Th}{\bf(Lipschitz-like property of solution maps implies their isolated calmness).}\label{ll-calm} Let $S_{KKT}$ be the solution map \eqref{skkt} of the canonically perturbed KKT system \eqref{kkt} with the piecewise linear-quadratic term \eqref{theta}, and let $(\xb,\olm)\in S_{KKT}(0,0)$. If $S_{KKT}$ is Lipschitz-like around $\big((0,0),(\xb,\olm)\big)$, then it enjoys the isolated calmness property at this point.
\end{Th}
\begin{proof} Assuming that $S_{KKT}$ has the Lipschitz-like property around $\big((0,0),(\xb,\olm)\big)$, we get implication \eqref{Lipdesc} by Theorem~\ref{Liplike}. On the other hand, we proceed similarly to the proof of Theorem~\ref{Liplike} and get counterparts of the equalities in \eqref{S} and \eqref{G} with replacing the coderivative by the graphical derivative therein.  The latter one is due to the easily checkable sum rule for graphical derivatives of summations with one smooth term as in \eqref{g}. Having this, we apply the Levy-Rockafellar criterion of isolated calmness \eqref{grdcr} to the solution map \eqref{skkt} and thus conclude that the isolated calmness of $S_{KKT}$ at $\big((0,0),(\xb,\olm)\big)$ is equivalent to
\begin{equation}\label{calmS}
\begin{cases}
\nabla_{xx}^2L(\xb,\olm)\xi+\nabla\Phi(\xb)^*\eta=0\\
\eta\in\big(D\partial\thy\big)(\Phi(\ox),\olm)\big(\nabla\Phi(\xb)\xi\big)
\end{cases}
\Longrightarrow(\xi,\eta)=(0,0).
\end{equation}
Comparing \eqref{Lipdesc} and \eqref{calmS}, we see that the only difference is in terms involving $(D^*\partial\thy)(\Phi(\xb),\olm)$ and $(D\partial\thy)(\Phi(\xb),\olm)$. To this end we use derivative-coderivative relationship from \cite[Theorem~13.57]{rw}, which tells us that the inclusion
\begin{equation*}
(D\partial\thy)(\Phi(\xb),\olm)(u)\subset(D^*\partial\thy)(\Phi(\xb),\olm)(u)\;\mbox{ for all }\;u\in\mathbb{R}^m
\end{equation*}
holds under the assumptions that are automatically satisfied for the piecewise linear-quadratic function $\thy$ from \eqref{theta}. This therefore completes the proof of the theorem.
\end{proof}

\end{document}